\documentclass[a4paper]{article}

\usepackage{tikz, subfigure, xcolor} 

\usepackage{amsmath}
\usepackage{amsthm}
\usepackage{amssymb}
\usepackage{amsfonts}

\usepackage{dsfont}

\usepackage{hyperref}

\newtheorem{theorem}{Theorem}[section]
\newtheorem{lemma}[theorem]{Lemma}
\newtheorem{proposition}[theorem]{Proposition}
\newtheorem{corollary}[theorem]{Corollary}
\theoremstyle{definition}
\newtheorem{definition}[theorem]{Definition}
\newtheorem{remark}[theorem]{Remark}
\newtheorem{example}{Example}[section]
\newenvironment{Example}{\begin{example}}{\hfill\qed\end{example}}

\numberwithin{equation}{section}

\newcommand{\diag}{\mathrm{diag}}      
\newcommand{\clo}{\mathrm{clo}}        
\newcommand{\Span}{\mathrm{span}}       
\renewcommand{\Im}{{\ensuremath{\mathrm{Im\,}}}} 
\renewcommand{\Re}{{\ensuremath{\mathrm{Re\,}}}} 

\providecommand{\norm}[1]{\lVert#1\rVert} 
\providecommand{\abs}[1]{\lvert#1\rvert} 

\DeclareMathOperator{\Ran}{Ran} 
\DeclareMathOperator{\ran}{Ran} 
\DeclareMathOperator{\Rank}{Rank} 
\DeclareMathOperator{\sign}{sign} 
\DeclareMathOperator{\sgn}{sign} 

\DeclareMathOperator{\Ker}{Ker}
\DeclareMathOperator{\Dom}{Dom}

\DeclareMathOperator{\Gr}{Gr}

\newcommand{\au}{\underline{a}}

\newcommand{\R}{\mathbb{R}}

\newcommand{\C}{\mathbb{C}}

\newcommand{\N}{\mathbb{N}}

\newcommand{\cH}{{\mathcal H}}

\newcommand{\cP}{{\mathcal P}}

\newcommand{\Me}{{\mathcal M}}
\newcommand{\He}{{\mathcal H}}
\newcommand{\Ke}{{\mathcal K}}
\newcommand{\Ge}{{\mathcal G}}
\newcommand{\Ie}{{\mathcal I}}
\newcommand{\Ee}{{\mathcal E}}
\newcommand{\De}{{\mathcal D}}

\newcommand{\We}{{\mathcal W}}

\newcommand{\cf}{\emph{cf.}}
\newcommand{\ie}{{\emph{i.e.}}}
\newcommand{\eg}{{\emph{e.g.}}}
\usepackage{eucal}
\newcommand{\verts}{{\mathcal V}}

\title{\textbf{Non-self-adjoint graphs}}

\author{
Amru Hussein$^{a}$, 
David Krej\v{c}i\v{r}\'{i}k$^{b}$ and 
Petr Siegl$^{c,} \footnote{On a leave from $b$.} $
}

\date{
\small
\emph{
\begin{quote}
\begin{itemize}
\item[$a)$]
Institut f\"{u}r Mathematik,
Johannes Gutenberg-Universit\"{a}t Mainz, 
Staud\-ing\-er Weg 9, 55099 Mainz, Germany; 
hussein@mathematik.uni-mainz.de
\item[$b)$]
Department of Theoretical Physics, Nuclear Physics Institute ASCR, 
25068 \v{R}e\v{z}, Czech Republic; 
krejcirik@ujf.cas.cz
\item[$c)$]
Mathematical Institute, University of Bern,
Sidlerstrasse 5, 3012 Bern, Switzerland; 
petr.siegl@math.unibe.ch
\end{itemize}
\end{quote}
}
\medskip
20 August 2013
}

\begin{document}

\maketitle

\begin{abstract}
\noindent
On finite metric graphs we consider Laplace operators,
subject to various classes of non-self-adjoint boundary conditions 
imposed at graph vertices. 
We investigate spectral properties,
existence of a Riesz basis of projectors
and similarity transforms to self-adjoint Laplacians.
Among other things, we describe a simple way
how to relate the similarity transforms between Laplacians 
on certain graphs with elementary similarity transforms between
matrices defining the boundary conditions.
\end{abstract}

\newpage
\tableofcontents

\newpage
\section{Introduction} 
%
The subject of differential operators on metric graphs has attracted 
a lot of attention in the last decades. This topic has become popular 
under the name ``quantum graphs'', referring to its background 
and applications in quantum mechanics. 
Since a quantum system is described by a unitary time evolution,
most of the literature has been concerned 
with \emph{self-adjoint} Schr\"odinger operators. 
For more details and many references, 
we refer to the surveys \cite[Chap.~17]{Exnerbook} and \cite{KuchBer} 
together with the articles \cite{PKWaves, PKQG1, Kurasov2}. 

In other areas of physics, where a system is described 
by non-conservative equations of motion,
it is necessary to deal with \emph{non-self-adjoint} operators.
As an example, let us mention 
stochastic processes on metric graphs \cite{VK2007, KPS2008, KPS2012}.
Furthermore, there have been recent attempts 
to develop ``quasi-Hermitian quantum mechanics'', 
where physical observables are represented by non-self-adjoint operators~$T$
satisfying the \emph{quasi-self-adjointness} relation
\begin{equation}\label{quasi}
  T^* = \Theta T \Theta^{-1}
\end{equation}
with a bounded, boundedly invertible and positive operator $\Theta=G^* G$.
The idea goes back to the paper~\cite{SGH} by nuclear physicists,
where~$\Theta$ is called \emph{metric},
since it defines a new inner product in the underlying Hilbert space
with respect to which~$T$ becomes self-adjoint. 
In other words, $T$~is similar to a self-adjoint operator via the similarity 
transformation~$G$, namely $GTG^{-1}$ is self-adjoint.
A consistent quantum theory can be built for quasi-self-adjoint operators.

It is not easy to decide whether a given non-self-adjoint operator
is quantum-mechanically admissible, \ie~quasi-self-adjoint.
A necessary condition for the quasi-self-adjointness of~$T$ 
is that its spectrum $\sigma(T)$ is real.
It was noticed that many operators commuting with an anti-unitary operator 
called \emph{symmetry} have the real spectrum.
This observation is behind the boom of the so-called
``$\mathcal{PT}$-symmetric quantum mechanics'' 
\cite{Bender07,Mostafazadeh},
which we use here as a source of 
interesting quasi-self-adjoint models. 
In this context, non-self-adjoint operators on metric graphs
were previously considered in \cite{Astudillo,Znojil}.

The present work is motivated by the growing interest 
in spectral theory on network structures 
and by the fresh relevance of non-self-adjoint operators 
in quantum mechanics.
We regard metric graphs as an intermediate step
between Sturm-Liouville operators on intervals
and partial differential operators.
Indeed, we shall be able to rigorously investigate some non-trivial
properties related to the spectrum and quasi-self-adjointness
that one can hardly expect to obtain in such a generality
in higher dimensions.

We restrict ourselves to a simple differential operator on the graph
-- namely the Laplacian --
but consider arbitrary non-self-adjoint
interface or boundary conditions on the graph vertices.
The standard material about Laplacians on metric graphs
is collected in the forthcoming Section~\ref{Sec.graphs}.
In a long Section~\ref{Sec.bc} divided into many subsections
we introduce various classes
of boundary conditions for the Laplacian.
The emphasis is not put on a systematic classification
of non-self-adjoint boundary conditions,
but rather on a diversity motivated by different applications
and on intriguing examples with wild spectra.

Spectral theory for the Laplacians is developed in Section~\ref{sec:spec}.
There we also present an explicit integral-type formula for the resolvent,
with a proof postponed to Appendix~\ref{Sec.App}. 
In Section~\ref{sec:riesz},
we apply an abstract result of Agranovich~\cite{Agr94}
to show that the eigensystem of a non-self-adjoint Laplacian 
on a compact metric graph contains a Riesz basis of subspaces.

Finally, in Section~\ref{sec:sim} we discover a simple way
how to relate the similarity transforms between Laplacians 
on graphs with elementary similarity transforms between
matrices defining the boundary conditions.
This main result enables us not only to effectively analyse 
the problem of quasi-self-adjointness for such graphs
but it turns out to be technically useful 
for self-adjoint Laplacians, too.

\section{The Laplacian on finite metric graphs}\label{Sec.graphs} 
%
Metric graphs are locally linear one dimensional spaces 
with singularities at the vertices, 
and one can think roughly of a metric graph as a union 
of finitely many finite intervals $[0,a_i]$, 
with $a_i \in (0,\infty)$, 
or semi-infinite intervals $[0,\infty)$
glued together at their end points. 
This intuitive picture is formalised here
by recalling from \cite{VKRS2006,VKScattering,KPS2008}
some notation and basic definitions.

\paragraph{Graph as a topological space}
A \emph{graph} is a $4$-tuple 
$\Ge = \left( \verts, \Ie,\Ee, \partial \right)$, 
where $\verts$ denotes the set of \textit{vertices}, 
$\Ie$ the set of \textit{internal edges} 
and $\Ee$ the set of \textit{external edges}, 
where the set $\Ee \cup \Ie$ is summed up in the notion \textit{edges}. 
The \textit{boundary map} $\partial$ assigns to each internal edge $i\in \Ie$ 
an ordered pair of vertices $\partial (i)=\left(\partial_-(i),\partial_+(i)\right)\in \verts \times \verts$, 
where $\partial_-(i)$ is called its \textit{initial vertex} 
and $\partial_+(i)$ its \textit{terminal vertex}. 
Each external edge $e\in \Ee$ is mapped by $\partial$ onto a single, 
its initial, vertex. 
The \textit{degree} $\deg(v)$ of a vertex $v\in \verts$ 
is the number of edges with initial vertex $v$ plus the number of edges 
with terminal vertex $v$. 
A graph is called \textit{finite} if $\abs{\verts}+\abs{\Ie}+\abs{\Ee}<\infty$ 
and a finite graph is called \textit{compact} if $\Ee=\emptyset$.

\paragraph{Graph as a metric space}
A graph $\Ge$ is endowed with the following metric structure. 
Each internal edge $i\in \Ie$ is associated 
with an interval $[0,a_i]$, with $a_i>0$, 
such that its initial vertex corresponds to~$0$ 
and its terminal vertex to~$a_i$. 
Each external edge $e\in \Ee$ is associated to the half line $[0,\infty)$ 
such that $\partial(e)$ corresponds to~$0$. 
The numbers~$a_i$ are called \textit{lengths} of the internal edges $i\in \Ie$ 
and they are summed up into the vector 
$\au=\{a_i\}_{i\in \Ie}\in (0,\infty)^{\abs{\Ie}}$. 
The $2$-tuple consisting of a finite graph endowed with a metric structure is called a \textit{metric graph} $(\Ge,\au)$. 
The metric on $(\Ge,\au)$ is defined via minimal path lengths. 

\paragraph{Graph as a measure space}
Equipping each edge of the metric graph with
the one-dimensional Lebesgue measure, 
we obtain a measure space. 
Any function $\psi \colon (\Ge,\au) \rightarrow \C$ can be written as 
\begin{eqnarray*}
\psi(x_j)= \psi_j(x), & \mbox{where} & \psi_j \colon I_j \rightarrow \C,
\end{eqnarray*}
with
\begin{equation*}
I_j= \begin{cases} [0,a_j], & \mbox{if} \ j\in \Ie, \\ [0,\infty), &\mbox{if} \ j\in \Ee.  \end{cases}
\end{equation*}
Occasionally we write also $\psi_j(x)=\psi_j(x_j)$. One defines 
\begin{equation*} 
\int_{\Ge} \psi := \sum_{i\in \Ie}  \int_{0}^{a_i}  \psi(x_i) \, dx_i +  \sum_{e\in \Ee}  \int_{0}^{\infty}  \psi(x_e)\, dx_e, 
\end{equation*}
where $dx_{i}$ and $dx_e$ refers to integration with respect to the Lebesgue measure on the intervals $[0,a_i]$ and $[0,\infty)$, respectively. 

\paragraph{Graph as a Hilbert space}
Given a finite metric graph $(\Ge,\au)$ one considers the Hilbert space
\begin{eqnarray*}
 \He \equiv \He(\Ee,\Ie,\au)= \He_{\Ee} \oplus \He_{\Ie}, & \displaystyle{\He_{\Ee}= \bigoplus_{e\in\Ee} \He_e,} & \He_{\Ie}= \bigoplus_{i\in\Ie} \He_i,
\end{eqnarray*}     
where $\He_j= L^2(I_j;\C)$. Hence, the scalar product in $\He$ is given by 
\begin{equation*} 
\langle \psi,\varphi\rangle =\int_{\Ge} \psi\, \overline{\varphi}. 
\end{equation*}

\paragraph{Graph as an energy space}
Denote by $\We_j$, $j\in \Ee \cup \Ie$ the set of all functions $\psi_j\in \He_j$ which are absolutely continuous with square integrable derivative $\psi_j^{\prime}$, and set
\begin{equation}\label{W-space}
  \We=\bigoplus_{j\in \Ee \cup \Ie} \We_j
  .
\end{equation}
With the scalar product defined by 
$$\langle \psi,\varphi\rangle_{\We}:= \langle \psi^{\prime},\varphi^{\prime}\rangle+\langle \psi,\varphi\rangle$$
the space $\We$ becomes a Hilbert space. 

By $\De_j$ with $j\in \Ee \cup \Ie$ denote the set of all $\psi_j\in \He_j$ such that $\psi_j$ and its derivative $\psi_j^{\prime}$ are absolutely continuous and its second derivative $\psi_j^{\prime\prime}$ is square integrable. Let $\De_j^0$ denote the set of all elements $\psi_j\in \De_j$ with
\begin{eqnarray*}
\psi_j(0)=0, &  \psi^{\prime}(0)=0, & \mbox{for}  \ j\in \Ee,    \\
\psi_j(0)=0, &  \psi^{\prime}(0)=0, & \psi_j(a_j)=0, \  \psi^{\prime}(a_j)=0, \ \mbox{for} \ j\in \Ie. 
\end{eqnarray*}
The sets
\begin{eqnarray*}
\De= \bigoplus_{j\in \Ee \cup \Ie} \De_j & \mbox{and} &  \De^0= \bigoplus_{j\in \Ee \cup \Ie} \De_j^0
\end{eqnarray*}
together with the scalar product defined by 
$$\langle \psi,\varphi\rangle_{\De}:= \langle \psi^{\prime\prime},\varphi^{\prime\prime}\rangle+\langle \psi,\varphi\rangle_{\We}$$
become Hilbert spaces, such that $\De^0\subset\De$ is closed.   

\paragraph{Graph as a Laplacian}
Let $\Delta$ be the differential operator
\begin{eqnarray*}
\left( \Delta \psi\right)_j (x) 
= \frac{d^2}{dx^2}\psi_j(x), & j\in \Ee\cup \Ie, & x\in I_j,
\end{eqnarray*}
with domain $\De$, 
and $\Delta^0$ its restriction on the domain $\De^0$. 
%
It is known that the operator $\Delta^0$ is a closed symmetric operator with deficiency indices $(d,d)$, where 
\begin{equation}\label{d-notation}
  d:=\abs{\Ee}+ 2\abs{\Ie}
  ,
\end{equation}
and its Hilbert space adjoint is $(\Delta^0)^*=\Delta$; 
see, \eg, \cite[Sec.~4.8]{Exnerbook}. 

Any closed extension $-\widetilde{\Delta}$ of $-\Delta^0$ satisfying
\begin{equation}\label{extensions}
\Delta^0 \subset \widetilde{\Delta} \subset \Delta
\end{equation}
will be called the \emph{Laplacian} on~$(\Ge,\au)$.
Self-adjoint Laplacians on graphs are well studied. 
The aim of this paper is to discuss extensions 
of $-\Delta^0$ which are not necessarily self-adjoint.

The extensions $\widetilde{\Delta}$ of $\Delta^0$ with~\eqref{extensions} 
can be discussed in terms of boundary or matching conditions imposed at the endpoints of the edges. For this purpose one defines for $\psi\in \De$ the vectors of boundary values
\begin{eqnarray*}
\underline{\psi}= \begin{bmatrix} \{\psi_{e}(0)\}_{e\in\Ee} \\ 
\{\psi_{i}(0)\}_{i\in\Ie} \\
\{\psi_{i}(a_i)\}_{i\in\Ie}
\end{bmatrix} &\mbox{and} & 
\underline{\psi}^{\prime}= \begin{bmatrix} \{\psi_{e}^{\prime}(0)\}_{e\in\Ee} \\ 
\{\psi_{i}^{\prime}(0)\}_{i\in\Ie} \\
\{-\psi_{i}^{\prime}(a_i)\}_{i\in\Ie}
\end{bmatrix}.
\end{eqnarray*}
One introduces the auxiliary Hilbert space
\begin{equation*}
\Ke \equiv \Ke(\Ee, \Ie) = \Ke_{\Ee}  \oplus \Ke_{\Ie}^- \oplus \Ke_{\Ie}^+
\end{equation*}
with $\Ke_{\Ee} = \C^{\abs{\Ee}}$ and $\Ke_{\Ie}^{(\pm)} = \C^{\abs{\Ie}}$. One sets 
\begin{equation*}
[\psi]:= \underline{\psi} \oplus \underline{\psi^{\prime}} \in \Ke \oplus \Ke.
\end{equation*}

Any extension $\widetilde{\Delta}$ with~\eqref{extensions}  
can be associated with a subspace $\Me\subset \Ke^2:=\Ke \oplus \Ke$ such that $\widetilde{\Delta}=\Delta(\Me)$ is the restriction of $\Delta$ to the domain
\begin{equation*}
\Dom(\Delta(\Me))= \{ \psi \in \De \mid [\psi] \in \Me   \}.
\end{equation*}

\section{Classification of boundary conditions}\label{Sec.bc}
%
There are various ways to parametrise the subspaces $\Me\subset \Ke^2$. In the following some parametrisations are discussed starting with self-adjoint boundary conditions and then transferring the methods to non-self-adjoint ones. 

Given linear maps $A,B$ in $\Ke$, one defines
\begin{eqnarray*}
(A, \, B)\colon \Ke^2 \rightarrow \Ke, & (A, \, B) (\chi_1 \oplus \chi_2) = A\chi_1 + B \chi_2 & \mbox{for } \chi_1,\chi_2\in\Ke,
\end{eqnarray*}
and sets
\begin{equation*}
\Me(A,B):=\Ker (A, \, B).
\end{equation*}
If $\dim \Me \geq d$ there are appropriate operators $A,B$ acting in $\Ke$ such that $\Me=\Me(A,B)$, and then an equivalent description of $\Dom (\Delta(\Me))$ is that it contains all functions $\psi\in \De$ satisfying the linear 
\emph{boundary conditions}
\begin{equation}\label{AB}
A \underline{\psi} + B \underline{\psi}^{\prime} = 0.
\end{equation}
In this case one also writes equivalently $\Delta(\Me)=\Delta(A,B)$. 
Note that the parametrisation by the matrices $A$ and $B$ is not unique. 
Indeed, operators $\Delta(A,B)$ and $\Delta(A^{\prime},B^{\prime})$ agree if and only if the corresponding spaces $\Me(A,B)$ and $\Me(A^{\prime},B^{\prime})$ agree. Therefore we introduce  
\begin{definition}
Boundary conditions defined by $A,B$ and $A^{\prime},B^{\prime}$ are called \textit{equivalent} if $\Me(A,B)=\Me(A^{\prime},B^{\prime})$.
\end{definition}

Notice that the boundary conditions 
are equivalent if and only if there exists 
an invertible operator $C$ in $\Ke$ such that simultaneously 
\begin{eqnarray*}
A^{\prime} = CA & \mbox{and} & B^{\prime} = CB.
\end{eqnarray*}

\subsection{Self-adjoint boundary conditions}
Recall that any self-adjoint realisation of $\Delta$ can be parametrised as $\Delta(A,B)$, where the matrices $A$ and $B$ satisfy
$AB^*=BA^*$ and $\dim \Me(A,B)=d$, 
where~$d$ is defined in~\eqref{d-notation}; 
see, \eg, \cite[Lem.~2.2 and below it]{VKRS1999}. 

It is a classical result that there is a one-to-one correspondence between unitary operators $U$ in $\Ke$ and self-adjoint realisations of $\Delta$. More precisely, any self-adjoint extension of $\Delta^0$ can be defined by the boundary conditions
\begin{equation}\label{U}
-\frac{1}{2}\left(U-\mathds{1} \right) \underline{\psi} + \frac{1}{2ik}\left(U+\mathds{1} \right)\underline{\psi}^{\prime} = 0,
\end{equation}
for $k>0$; 
see, \eg, \cite[Sec.~3]{Harmer}.

The link between the parametrisation by unitary operators $U$ and the one by matrices $A$ and $B$ in \eqref{AB} is given by a Cayley transform. 
For $A,B$ defining a self-adjoint Laplacian, 
consider, for $k\in \C\setminus \{0\}$ such that $A+ik B$ is invertible, 
the transform 
\begin{equation}\label{S}
\mathfrak{S}(k,A,B):= - \left(A+ik B \right)^{-1}\left(A-ik B \right).
\end{equation}
For $k>0$ the operator $\mathfrak{S}(k,A,B)$ is unitary
\cite[Thm.~2.1]{VKRS1999}
and one can choose $U=\mathfrak{S}(k,A,B)$ in \eqref{U}, 
\cf~\cite[p.~209]{VKRS2006}. 
For  self-adjoint Laplacians on graphs with $\Ie=\emptyset$ the matrix $\mathfrak{S}(A,B,k)$ admits also the interpretation as the \textit{scattering matrix} 
for a certain scattering pair~\cite{VKScattering}.

\subsection{Regular boundary conditions}\label{subsec:regbc}
The transform $\mathfrak{S}(k,A,B)$ can be defined for non-self-adjoint boundary conditions as well whenever $A+ikB$ is invertible, and then $\mathfrak{S}(k,A,B)$ is independent of the concrete choice of $A,B$ representing $\Me=\Me(A,B)$. So, whenever $A+ikB$ is invertible for $k\in \C \setminus \{0\}$ one re-obtains from $\mathfrak{S}(k,A,B)$ equivalent boundary conditions of the form \eqref{AB} by
\begin{eqnarray}\label{AS}
A_{\mathfrak{S}}:= - \frac{1}{2} \left(\mathfrak{S}(k,A,B) -\mathds{1}\right) & \mbox{and} & B_{\mathfrak{S}}:=  \frac{1}{2ik} \left(\mathfrak{S}(k,A,B) +\mathds{1}\right).
\end{eqnarray}
This follows from the equalities
\begin{eqnarray*}
(A+ikB) A_{\mathfrak{S}}=A & \mbox{and} &  (A+ikB)B_{\mathfrak{S}}=  B 
\end{eqnarray*}
used in \cite[proof of Lem.~3.4]{KPS2008}. 
A necessary condition for the definition of $\mathfrak{S}(k,A,B)$ is that $\dim \Me(A,B)=d$, but this is not sufficient.
Actually, since $\det(A+ik B)$ is a polynomial in $k$ of degree at the most $d$, $A+ik B$ is not invertible either for every $k\in \C$ or only for finitely many values $k\in\C$. 
\begin{definition}\label{Def.regular}
Boundary conditions~\eqref{AB} defined by $A,B$ 
with $\dim\Me(A,B)=d$ such that $A+ik B$ is invertible 
for some $k\in\C$ 
are called \emph{regular} boundary conditions. 
\end{definition}

\subsection{Other notions of regular boundary conditions}\label{rem:Naimark}
The reader is warned that there exist 
further parametrisations and classifications of boundary conditions 
for the second derivative operator acting on intervals. 
For instance, the classification given in \cite[Sec.~XIX.4]{DSIII} 
is based on the structure of certain determinants related to the secular equation, 
and this gives rise to an alternative regularity assumption
\cite[Hypothesis XIX.4.1]{DSIII} on boundary conditions. 
The aim in \cite{DSIII} is to define \emph{spectral operators} 
and the regularity hypothesis goes back to \cite{Birkhoff1, Birkhoff2}.

That the regularity hypothesis formulated in \cite[Hypothesis XIX.4.1]{DSIII} does not agree with the notion of regular boundary conditions introduced 
in our Definition~\ref{Def.regular}
follows already from \cite[Ex.~XIX.6(d)]{DSIII}, 
which is discussed here as Example~\ref{exDS} below. 
The boundary conditions given in \cite[Ex.~XIX.6(d)]{DSIII} 
are called \emph{intermediate} boundary conditions 
and are an example of a class of boundary conditions 
not satisfying the regularity hypothesis \cite[Hypothesis XIX.4.1]{DSIII}, see also \cite[p.383]{Birkhoff2}, whereas they are regular in the sense introduced here.

In general it seems difficult to make precise statement on the secular equation for -- in our sense -- regular boundary conditions. More generally, when considering non-compact graphs, \ie~$\Ee\neq \emptyset$, 
there is no straightforward generalisation of the regularity hypothesis of 
\cite[Sec.~XIX.4]{DSIII} since it is dealing with operators with discrete spectrum. 

\subsection{Irregular boundary conditions}
Boundary conditions defined by $A,B$ with $\dim\Me(A,B)=d$ 
which are not regular will be called \emph{irregular}. 
We do not include the situations $\dim\Me(A,B)\not=d$
into our notion of irregular boundary conditions,
since they are not spectrally interesting.
Indeed, it follows from Proposition~\ref{prop:dimM} below
that $\sigma(-\Delta(A,B))=\C$ whenever $\dim\Me(A,B)\not=d$.

The class of regular boundary conditions covers many relevant and interesting cases, whereas the irregular boundary conditions seem 
to be rather pathological. 
Indeed, the latter are typically associated with operators 
that have empty resolvent set or empty spectrum,
even if $\dim\Me(A,B)=d$ holds.

\begin{Example}[Indefinite Laplacian, no resolvent set]\label{sgnsgn}
Consider the boundary conditions~\eqref{AB} given by 
\begin{eqnarray*}
A=\begin{pmatrix} 1 & -1 \\ 0 & 0\end{pmatrix} 
&\mbox{and} & B=\begin{pmatrix}  0 & 0 \\ 1 & -1 \end{pmatrix}
\end{eqnarray*} 
for the graph  $\Ge=(V,\partial,\Ee)$ consisting of 
two external edges $\Ee=\{e_1,e_2\}$ and one vertex $\partial(e_1)=\partial(e_2)$. Identifying this graph with the real line, 
the operator $-\Delta(A,B)$ corresponds to the indefinite operator 
$$
  -\sign(x)\frac{d}{dx}\sign(x)\frac{d}{dx}
  \qquad\mbox{on}\qquad
  L^2(\R)
$$ 
with its natural domain
$
  \{\psi \in W^{1,2}(\R) \,|\, (\psi'\sign)' \in L^2(\R)\}
$. 
This operator is studied 
within the framework of Krein space theory in \cite[Sec.~5]{Kuzel}.

This example demonstrates in particular 
that $\dim \Me(A,B) =d$ is a necessary 
but not a sufficient condition for $A,B$ 
to define regular boundary conditions. 
Indeed, $\dim \Me(A,B)=2=d$ in this example,
while $A+ik B$ is invertible for no complex~$k$.
(As a consequence, 
the statement in \cite[observation below Ass.~2.1]{KPS2008} 
is not correct in general, 
but it holds for the boundary conditions defining 
m-accretive operators studied there.) 

Note that the equation $\det (A+ik B)=0$ with $\Im k>0$ 
is the secular equation for the spectral problem 
associated with $-\Delta(A,B)$, 
\cf~Subsection~\ref{subsec:ev} below. Therefore the spectrum of the operator described in the present example is entire~$\C$. 
This fact will be explained also in Subsection~\ref{subsec:irr} 
by means of a similarity transform.
\end{Example}

\begin{Example}[Totally degenerate boundary conditions, 
no spectrum]\label{ex:emptyspec}
This example is overtaken from \cite[Sec.~XIX.6(b)]{DSIII}.
Consider the interval $[0,1]$ and the irregular boundary conditions defined by 
\begin{eqnarray*}
A=\begin{bmatrix}1 & 0 \\ 0 & 0  \end{bmatrix}  & \mbox{and} & B=\begin{bmatrix} 0 & 0 \\ 1 & 0  \end{bmatrix}.
\end{eqnarray*}
Then $\dim\Me(A,B)=2=d$ and the boundary conditions correspond to
\begin{eqnarray*}
\psi(0)=0 & \mbox{and} & \psi^{\prime}(0)=0,
\end{eqnarray*}
whereas on the other endpoint no boundary conditions are imposed. By integration one can show that this operator is boundedly invertible, and for the compact embedding $\De \hookrightarrow \He$ the inverse is compact, and hence the operator $-\Delta(A,B)$ has only point spectrum. However, a direct computation shows that for these boundary conditions there are no eigenvalues, and therefore the spectrum of $-\Delta(A,B)$ is empty.    
\end{Example}

By inspection of the previous examples, 
it is straightforward to identify the mechanism
which is behind the irregularity of the boundary conditions.
\begin{proposition}\label{KerAB}
Let $A,B$ be maps in $\Ke$ such that $\dim \Me(A,B)=d$. Then $A,B$ define irregular boundary conditions if and only if 
\begin{eqnarray*}
\Ker A \cap \Ker B\neq \{0\}.  
\end{eqnarray*} 
\end{proposition}
\begin{proof}
If $\Ker A \cap \Ker B\neq \{0\}$, 
then for a non-zero 
$\underline{\psi}\in \Ker A \cap \Ker B$ 
one has $(A+ik B)\underline{\psi}=0$ for any $k\in \C$. 
The other way round, if $A+ik B$ is not invertible for any $k\in \C$, then $\Ker A\neq \{0\}$, since otherwise one could consider $\mathds{1}+ik A^{-1}B$ which is invertible for $k$ sufficiently small. 
So, for non-zero $\underline{\psi}\in \Ker A$ 
one has $ik B\underline{\psi}=0$ 
for any $k\in \C$ and hence $\underline{\psi}\in\Ker B$, 
which proves that $\Ker A \cap \Ker B\neq \{0\}$. 
\end{proof}

A possible generalisation of Example~\ref{sgnsgn} 
to more complex graphs is given by the following example.
\begin{Example}[A generalisation of Example~\ref{sgnsgn}]\label{gsgnsgn}
Consider a star graph with $\Ie=\emptyset$ 
and subdivision of the external edges $\Ee=\Ee_+ \dot{\cup} \Ee_-$
together with the boundary conditions defined by
\begin{align*}
A= \left[
   \begin{array}{cccccc}
     1 & -1 & 0 &\cdots & 0 & 0 \\
     0 & 1 & -1 &\cdots & 0 & 0  \\
     0 & 0 & 1 &\cdots & 0 & 0  \\
     \vdots &\vdots  & \vdots & & \vdots  & \vdots \\
        0 & 0 & 0 &\cdots & 1 & -1  \\
   0 & 0 & 0 &\cdots & 0 & 0 
   \end{array}
\right], && B= \left[
   \begin{array}{cccccc}
     0 & 0 & 0 &\cdots & 0 & 0 \\
     0 & 0 & 0 &\cdots & 0 & 0  \\
     0 & 0 & 0 &\cdots & 0 & 0  \\
     \vdots &\vdots  & \vdots & & \vdots  & \vdots \\
        0 & 0 & 0 &\cdots & 0 & 0  \\
   1 & \cdots & 1 & \cdots & -1 & -1 
   \end{array}
\right],
\end{align*}  
where in the last row of~$B$ for each edge in $\Ee_+$ stands a $+1$ 
and for each edge in $\Ee_-$ a $-1$. 
To paraphrase, these boundary conditions guarantee that functions are continuous at the central vertex and that the sum of the outward directed derivatives evaluated at the positive incident edges equals the sum of the outward directed derivatives evaluated at the negative incident edges. These boundary conditions define an operator which is self-adjoint in a certain Krein space~\cite[Sec.~4]{IndefiniteQG}. 
The kernel of $A$ is spanned by the vector $w$ with $(w)_{i}=1$, 
for all $i\in \{1, \ldots , |\Ee|\}$. 
Hence, by Proposition \ref{KerAB}, for $\abs{\Ee_+}=\abs{\Ee_-}$ the boundary conditions defined by $A,B$ are irregular and for $\abs{\Ee_+}\neq \abs{\Ee_-}$ they are regular. 
For example, in the case $\abs{\Ee_+}=2$ and $\abs{\Ee_-}=1$ 
one obtains the $k$-independent ``scattering matrix''
\begin{eqnarray*}
\mathfrak{S}(k, A,B)= \begin{bmatrix}
1 & 2 & -2 \\ 2 & 1 & -2 \\ 2 & 2 & -3
\end{bmatrix}.
\end{eqnarray*}
\end{Example}

\begin{remark}
Let the boundary conditions be \emph{local}, \ie $$\Me=\bigoplus_{v\in \verts} \Me_{v},$$ where $\Me_{v}$ are subspaces of $\Ke_{v}^2$, the space of boundary values associated with the endpoints of the edges incident in the vertex $v$, \cf~\cite[Def. 2.6]{KPS2008}. Then it is a direct consequence of Proposition~\ref{KerAB} that the boundary conditions are regular if and only if the boundary conditions at each vertex are regular, and irregular if at least at one vertex the boundary conditions are irregular. 
\end{remark}

\subsection{m-sectorial boundary conditions}
In \cite[Corol.~5]{PKQG1} a further way 
how to parametrise self-adjoint Laplacians on graphs is proposed.
It is given in terms of an orthogonal projection $P$ acting in $\Ke$ 
and a self-adjoint operator $L$ acting in the subspace $\Ker P$. 
For any self-adjoint Laplacian one has $-\Delta(A,B)=-\Delta(A^{\prime},B^{\prime})$ with $A^{\prime}= L+P$ and $B^{\prime}=P^{\perp}$, where using $\Ran B^{\ast} = (\Ker B)^{\perp}$ one sets
\begin{align*}
L= \left(B\mid_{\ran B^*}\right)^{-1}A P^{\perp}
\end{align*}
and $P$ denotes the orthogonal projector onto $\Ker B\subset \Ke$ and $P^{\perp}= \mathds{1}-P$ is the complementary projector. This parametrisation is unique in contrast to that using matrices $A,B$, 
and additionally it is convenient 
when considering forms associated with operators, 
\cf~\cite[Thms.~6 and 9]{PKQG1}.  

Inspired by the self-adjoint situation, 
for a given projector $P$ 
and a not necessarily self-adjoint operator~$L$ 
acting in $\Ker P$, \ie\ $L=P^{\perp}LP^{\perp}$, 
let us consider $-\Delta(\Me)$ with $\Me=\Me(L+P,P^{\perp})$. 
According to \cite[Thm.~3.1]{AccQG},
this operator is m-sectorial 
and associated with the closed sectorial form $\delta_{P,L}$ 
defined by 
\begin{equation}\label{deltaPL}
\begin{aligned}
\delta_{P,L}[\psi]
&=\int_{\Ge} \abs{ \psi^{\prime}}^2 
- \langle L P^{\perp}\underline{\psi},P^{\perp}\underline{\psi}\rangle_{\Ke}, 
\\
\psi \in \Dom(\delta_{P,L}) 
&= \{ \varphi\in \We \mid P\underline{\varphi}=0\},
\end{aligned}
\end{equation}
where~$\We$ denotes the Sobolev space~\eqref{W-space}.

The question when $\Me(A,B)$ with $\dim \Me(A,B)=d$ admits an equivalent parametrisation in terms of a projector $P$ and an operator $L$ acting in $\Ker P$ such that $\Me(A,B)=\Me(L+P,P^{\perp})$ is discussed in \cite{AccQG}. 
It turns out that this is possible 
if and only if $-\Delta(A,B)$ is m-sectorial.
Furthermore, if $\Me(A,B)$ does not admit such a parametrisation, 
then the numerical range of $-\Delta(A,B)$ is entire $\C$, 
see \cite[Lem.~4.3]{AccQG}. 
Therefore, here we call boundary conditions 
defined by~$P$ and~$L$ as described above \textit{m-sectorial}.
Descriptive examples of such boundary conditions 
are $\delta$-interactions with generally complex coupling parameters. 
Note that in order to apply any kind of form methods 
one needs at least m-sectorial boundary conditions. 

\begin{Example}[Complex $\delta$-interaction]\label{Cdelta}
Consider a graph with $\Ie=\emptyset$ and $\abs{\Ee}\geq 2$. Assume that the boundary conditions are defined up to equivalence by 
\begin{align*}
A= \left[
   \begin{array}{cccccc}
     1 & -1 & 0 &\cdots & 0 & 0 \\
     0 & 1 & -1 &\cdots & 0 & 0  \\
     0 & 0 & 1 &\cdots & 0 & 0  \\
     \vdots &\vdots  & \vdots & & \vdots  & \vdots \\
        0 & 0 & 0 &\cdots & 1 & -1  \\
   -\gamma & 0 & 0 &\cdots & 0 & 0
   \end{array}
\right], && B= \left[
   \begin{array}{cccccc}
     0 & 0 & 0 &\cdots & 0 & 0 \\
     0 & 0 & 0 &\cdots & 0 & 0  \\
     0 & 0 & 0 &\cdots & 0 & 0  \\
     \vdots &\vdots  & \vdots & & \vdots  & \vdots \\
        0 & 0 & 0 &\cdots & 0 & 0  \\
   1 & 1 & 1 &\cdots & 1 & 1 
   \end{array}
\right],
\end{align*}  
where $\gamma\in \C$. For real $\gamma$ one can represent the boundary conditions equivalently by the m-sectorial boundary conditions defined by $P=\mathds{1}-P^{\perp}$, where $P^{\perp}$ is the rank one projector onto $(\Ker B)^{\perp}$, and $L= -\frac{\gamma}{\abs{\Ee}} P^{\perp}$, 
\cf~\cite[Sec.~3.2.1]{PKQG1}. 
A direct calculation shows that this carries over to the case of complex coupling parameters~$\gamma$. 
The operator $-\Delta(A,B)$ is associated 
with the quadratic form defined by 
\begin{eqnarray*}
\delta_{P,L}[\psi]=
\int_{\Ge} \abs{\psi^{\prime}}^2   
+ \frac {\gamma}{\abs{\Ee}} \abs{\underline{\psi}}^2, 
& \psi \in \Dom(\delta_{P,L}) 
= \{\psi\in \We \mid P\underline{\psi}=0\}.
\end{eqnarray*}
\end{Example}

It is proved in~\cite{AccQG} that
the boundary conditions of the form $\eqref{AB}$ 
defined by matrices $A,B$ 
can be substituted by an equivalent parametrisation 
using m-sectorial boundary conditions if and only if 
\begin{eqnarray}\label{QAP}
\dim\Me(A,B)=d & \mbox{and} & QAP^{\perp}=0,
\end{eqnarray}
where $Q$ is the orthogonal projector onto $(\Ran B)^{\perp}$ 
and $P^{\perp}$ the orthogonal projector onto $(\Ker B)^{\perp}$. 
This is due to the fact that the evaluation at the vertices 
of the derivatives cancel out in the corresponding quadratic form 
if and only if $QAP^{\perp}=0$.

Note that $L+P+ikP^{\perp}$ has a block diagonal form with respect to the decomposition of $\Ke$ into $\Ran P$ and $\Ran P^{\perp}$. Thus $L+P+ikP^{\perp}$ is invertible for $\abs{k}>\norm{L}$. Consequently, the parametrisation by the transform $\mathfrak{S}(k,L+P,P^{\perp})$ is admissible, whereas the converse is not true: from $A+ik B$ invertible, in general, it does not follow that there are equivalent m-sectorial boundary conditions. 
This is illustrated by the following examples.
\begin{Example}[From Kirchhoff 
to wild aperiodic boundary conditions]\label{ex1}
Let $\Ge=(\verts,\partial,\Ee)$ be a graph consisting of two external edges $\Ee=\{e_1,e_2\}$ and one vertex $\partial(e_1)=\partial(e_2)$. Consider the boundary conditions defined by 
\begin{eqnarray*}
A_{\tau}=\begin{pmatrix} 1 & -e^{i\tau} \\ 0 & 0\end{pmatrix} &\mbox{and} & B_{\tau}=\begin{pmatrix}  0 & 0 \\ 1 & e^{-i\tau} \end{pmatrix},
\end{eqnarray*} 
for $\tau \in [0,\pi/2]$. 
Identifying the graph with the real line 
and the vertex with zero,
the boundary conditions correspond to
$$
  \psi(0+) = e^{i\tau} \psi(0-)
  \quad\mbox{and}\quad
  \psi^{\prime}(0+) = e^{-i\tau} \psi^{\prime}(0-)
  .
$$
This example is included in the study of 
$\mathcal{PT}$-symmetric point interactions in~\cite{Kurasov02}
and was further investigated in \cite{Petr} and \cite{Kuzel05}.

The matrix $A_\tau + ik B_{\tau}$ 
is invertible for $\tau\in[0,\pi/2)$ and $k\neq 0$, hence $A_{\tau},B_{\tau}$ define regular boundary conditions for $\tau\in[0,\pi/2)$. For the Cayley transform
\begin{eqnarray*}
\mathfrak{S}(A_{\tau},B_{\tau},k)= - (A_{\tau}+ ik B_{\tau})^{-1}(A_{\tau}- ik B_{\tau}), & \tau \in [0,\pi/2),
\end{eqnarray*}
an explicit computation yields the $k$-independent matrix
\begin{eqnarray*}
\mathfrak{S}(A_{\tau},B_{\tau},k)= \frac{1}{\cos(\tau)} \begin{bmatrix} i \sin(\tau) & 1 \\ 1 & -i\sin(\tau) \end{bmatrix}.
\end{eqnarray*}
The operator $\mathfrak{S}(A_{\tau},B_{\tau},k)$ is unitary 
(with eigenvalues $+1$ and $-1$) only for $\tau=0$, 
where it defines the so-called standard or Kirchhoff boundary conditions. 

On the other hand, 
for $\tau = \pi/2$ one has $\det(A_{\pi/2}+ik B_{\pi/2})=0$ 
for any $k\in\C$, 
and therefore $A_{\pi/2},B_{\pi/2}$ define irregular boundary conditions. Furthermore one has 
$\sigma_\mathrm{p}(-\Delta(A_{\pi/2}, B_{\pi/2}))=\C\setminus[0,\infty)$, because of \eqref{eq:detZLB} below. This reproduces the results obtained in \cite[Thm.~2]{Kuzel05} and \cite[Sec.~3]{Petr}.

Explicit computation yields
\begin{eqnarray*}
\Ran B_{\tau} = \Span \left\{  \begin{pmatrix} 0 \\ 1  \end{pmatrix}\right\},
& & (\Ran B_{\tau})^{\perp} = \Span \left\{  \begin{pmatrix} 1 \\ 0  \end{pmatrix}\right\}, \\ 
 \Ker B_{\tau} = \Span \left\{  \begin{pmatrix} 1 \\ -e^{i\tau}  \end{pmatrix}\right\}, & &
 (\Ker B_{\tau})^{\perp} = \Span \left\{  \begin{pmatrix} 1 \\ e^{i\tau}  \end{pmatrix}\right\},
\end{eqnarray*}
and therefore, 
with $Q_{\tau}$ being the orthogonal projector 
onto $(\Ran B_{\tau})^{\perp}$ 
and $P_{\tau}^{\perp}$ being the orthogonal projector 
onto $(\Ker B_{\tau})^{\perp}$, 
one has   
\begin{eqnarray*}
Q_{\tau}A_{\tau}P_{\tau}^{\perp}= \frac{1}{2} \begin{pmatrix} 1-e^{2i\tau} & e^{-i\tau}-e^{i\tau} \\ 0 & 0   \end{pmatrix} \neq 0, & \mbox{for } \tau \in (0,\pi/2].
\end{eqnarray*}
The criterion in \eqref{QAP} implies that for $\tau\in(0,\pi/2]$ there is no equivalent representation of $A_{\tau},B_{\tau}$ by m-sectorial boundary conditions. This can be illustrated also by considering the quadratic form defined by the operator $-\Delta(A_{\tau},B_{\tau})$ which  by integrating by parts and inserting the boundary conditions simplifies to become
\begin{equation*}
\langle -\Delta(A_{\tau},B_{\tau})\psi,\psi\rangle = \int_{\Ge} \abs{\psi^{\prime}}^2 + (1-e^{2i\tau}) \psi_2(0)\overline{\psi_2^{\prime}(0)}
\end{equation*}
for every $\psi\in \Dom(-\Delta(A_{\tau},B_{\tau}))$.
In particular, the derivative term cannot be avoided, and the numerical range is entire $\C$ for all $\tau\in(0,\pi/2]$. 

Despite of the wild numerical range properties,
in Section~\ref{Sec.similar}
we shall show 
that for $\tau \in [0,\pi/2)$ the operator $-\Delta(A_{\tau},B_{\tau})$ 
is similar to the self-adjoint Laplacian $-\Delta(A_{0},B_{0})$ , 
and hence its spectrum is $[0,\infty)$. 
Such a similarity relation is of course impossible
for $\tau=\pi/2$ because the spectrum is entire~$\C$. 

The analogous operator on the graph with two internal edges of the same length defined by boundary conditions $A_{\tau},B_{\tau}$ at the central vertex and Dirichlet boundary conditions at the endpoints exhibit similar 
pathological behaviours, see~\cite[Sec.~3]{Petr}.
\end{Example}

\begin{Example}[Intermediate boundary conditions]\label{exDS}
Consider the interval $[0,1]$ 
and the regular boundary conditions defined by 
\begin{eqnarray*}
A=\begin{bmatrix}1 & 0 \\ 0 & 1  \end{bmatrix}  & \mbox{and} & B=\begin{bmatrix} 0 & 0 \\ -1 & 0  \end{bmatrix},
\end{eqnarray*}
\ie \ $\psi(0)=0$ and $\psi(1)-\psi^{\prime}(0)=0$. 
Then $\dim\Me(A,B)=2$ and
\begin{eqnarray*}
QAP^{\perp}=\begin{bmatrix}1 & 0 \\ 0 & 0  \end{bmatrix}\neq 0. 
\end{eqnarray*}
Hence the boundary conditions are not m-sectorial.
One has
\begin{equation*}
\langle -\Delta(A,B)\psi,\psi\rangle 
= \int_{0}^1 \abs{\psi^{\prime}}^2 
-\psi^{\prime}(0)\overline{\psi^{\prime}(1)}
\end{equation*}
for every $\psi\in \Dom(-\Delta(A,B))$.
This example can be found in \cite[p.383]{Birkhoff2} as well as in \cite[Ex.~XIX.6(d)]{DSIII},
where the boundary conditions are called \emph{intermediate}.

Using the methods developed in the forthcoming Section~\ref{sec:spec} one can show that the spectrum of $-\Delta(A,B)$ consists only of eigenvalues of geometric multiplicity one, where each eigenvalue is a solution of 
$\sin(k)=k$, $k\in \C$. 
\end{Example}

\subsection{Adjoint boundary conditions}
Consider for $\Me\subset \Ke^2$ the possibly non-self-adjoint operator $\Delta(\Me)$. Since $\Delta^0\subset \Delta(\Me)\subset \Delta$ it follows for the adjoint operator that $\Delta(\Me)^{\ast}=\Delta(\Me^{\ast})$ for an appropriate subspace $\Me^{\ast}\subset\Ke^2$, and hence also the adjoint operator can be described by means of boundary conditions.

\begin{proposition}\label{thm:adj}
Let $\Me \subset \Ke^2$, then $\Delta(\Me)^{\ast}= \Delta(\Me^{\ast})$ with 
\begin{eqnarray*}
\Me^{\ast}= \left( J\Me\right)^{\perp}, & \mbox{where } J=\begin{bmatrix} 0 & \mathds{1}_{\Ke} \\ - \mathds{1}_{\Ke} & 0 \end{bmatrix}
\end{eqnarray*}
defines a map in $\Ke^2$.
\end{proposition}
\begin{proof}
By definition, the adjoint of $\Delta(\Me)$ in the Hilbert space $\He$ 
is the operator defined on
\begin{equation*}
\Dom(\Delta(\Me)^{\ast})
=\left\{\psi \in \He \mid \exists \varphi \in \He , \,
\forall \eta\in \Dom(\Delta(\Me)), \
\langle \psi,\Delta(\Me)\eta\rangle= \langle \varphi,\eta\rangle \right\},
\end{equation*}
by $\Delta(\Me)^{\ast}\psi=\varphi$.
It follows from~\eqref{extensions} 
that $\Delta^0\subset \Delta(\Me)^{\ast}\subset \Delta$, 
and hence $\Delta(\Me)^{\ast}$ is also a realisation of~$\Delta$ 
defined by means of boundary conditions.
Consequently,
$\Delta(\Me)^{\ast}\psi=\Delta\psi$ 
and $\De^0\subset \Dom(\Delta(\Me)^{\ast})\subset \De$. 
It remains to determine the domain of $\Delta(\Me)^{\ast}$
by specifying the boundary conditions. 
An integration by parts yields
\begin{equation}\label{boundaryform}
\langle \psi, \Delta(\Me) \eta \rangle - \langle \Delta\psi, \eta \rangle 
= \langle J [\psi], [\eta] \rangle_{\Ke^2}
\end{equation}
for every $\eta\in \Dom(\Delta(\Me))$ and $\psi \in\De$.
Define   
\begin{eqnarray*}
[\cdot]_{\Me}\colon \Dom(\Delta(\Me)) \rightarrow \Ke^2, & [\eta]_{\Me}=[\eta],
\end{eqnarray*}
and observe that the range of $[\cdot]_{\Me}$ is $\Me$, and that the boundary term \eqref{boundaryform} vanishes identically for all $\psi \in \Dom(\Delta(\Me^{\ast}))$. Hence $\Delta(\Me^{\ast})\subset \Delta(\Me)^{\ast}$. 
Noticing that the boundary term in \eqref{boundaryform} 
vanishes for all $\eta\in \Dom(\Delta(\Me))$ 
if and only if $J[\psi]\perp \Me$ 
and using that $J$ is unitary,
we have $[\psi]\perp J \Me$ for $\psi\in \Dom(\Delta(\Me)^{\ast})$.
Consequently, $\Delta(\Me)^{\ast}\subset\Delta(\Me^{\ast})$,
which proves the claim.
\end{proof}

\begin{corollary}\label{cor:2d}
Let $\Me\in \Ke^2$, then $\dim \Me + \dim \Me^{\ast}=2d$.
\end{corollary}
\begin{proof}
As $J$ is unitary one has $\dim \Me = \dim J \Me$, from which the claim follows.
\end{proof}

\subsection{Adjoints for regular boundary conditions}
Searching for boundary conditions that define Laplacians 
with non-empty resolvent set, 
one needs by the forthcoming Proposition~\ref{prop:dimM} 
the condition $\dim \Me = d$, and therefore this case is of particular interest. For regular boundary conditions the parametrisation 
in terms of $\mathfrak{S}(k,A,B)$ is convenient for this purpose. 
\begin{proposition}\label{prop:adj}
Let $A,B$ be such that $\dim \Me(A,B)=d$ and $A+ik B$ is invertible for the number $k\neq 0$. Then an equivalent parametrisation of $\Delta(A,B)$ is given by 
\begin{eqnarray*}
A_{\mathfrak{S}}:= - \frac{1}{2} \left(\mathfrak{S}(k,A,B) -\mathds{1}\right) & \mbox{and} & B_{\mathfrak{S}}:=  \frac{1}{2ik} \left(\mathfrak{S}(k,A,B) +\mathds{1}\right)
\end{eqnarray*}
using $\mathfrak{S}(k,A,B)=-(A+ikB)^{-1}(A-ikB)$, and the adjoint operator $\Delta(\Me)^{\ast}=\Delta(\Me^{\ast})$ is defined by $\Me^{\ast}=\Me(A^{\prime},B^{\prime})$, where
\begin{eqnarray*}
A^{\prime}:= - \frac{1}{2} \left(\mathfrak{S}(k,A,B)^{\ast} -\mathds{1}\right) & \mbox{and} & B^{\prime}:=  \frac{1}{-2i\overline{k}} \left(\mathfrak{S}(k,A,B)^{\ast} +\mathds{1}\right).
\end{eqnarray*}
\end{proposition}
\begin{proof}
The fact that $A_{\mathfrak{S}}, B_{\mathfrak{S}}$ 
define equivalent boundary conditions has been discussed already 
in Subsection~\ref{subsec:regbc}.

Let us first prove that $\dim \Me(A^{\prime},B^{\prime})=d$ 
for the operators $A^{\prime},B^{\prime}$ given in the proposition. 
Assume that $\dim \Me(A^{\prime},B^{\prime})>d$. Then 
\begin{eqnarray*}
\Me(A^{\prime},B^{\prime})^{\perp}=\Ran\begin{bmatrix} (A^{\prime})^{\ast} \\ (B^{\prime})^{\ast}  \end{bmatrix} & \mbox{and} & \dim \Me(A^{\prime},B^{\prime})^{\perp} <d.
\end{eqnarray*}
Therefore, 
\begin{eqnarray*}
\Ker \begin{bmatrix} (A^{\prime})^{\ast} \\ (B^{\prime})^{\ast}  \end{bmatrix}=\Ker (A^{\prime})^{\ast} \cap \Ker (B^{\prime})^{\ast}\neq \{0\},
\end{eqnarray*}
\cf~\cite[Ass.~2.1 and below]{KPS2008}. 
Note that 
\begin{eqnarray*}
A^{\prime}= A_{\mathfrak{S}}^{\ast} & \mbox{and} & B^{\prime}= B_{\mathfrak{S}}^{\ast},
\end{eqnarray*}
and hence $\Ker A_{\mathfrak{S}} \cap \Ker B_{\mathfrak{S}}\neq \{0\}$, which implies $\det(A_{\mathfrak{S}}+ik B_{\mathfrak{S}})= 0$ for all $k\in\C$. This is a contradiction to the assumption that $A,B$ define regular boundary conditions. Hence $\dim \Me(A^{\prime},B^{\prime})=d$. 

Now one shows that $\Me^{\ast}=\Me(A^{\prime},B^{\prime})$, where $\Me^{\ast}$ is given in Proposition~\ref{thm:adj}. By the equivalence of boundary conditions, one has $\Me(A,B)=\Me(A_{\mathfrak{S}}, B_{\mathfrak{S}})$. Note that $J\Me(A_{\mathfrak{S}}, B_{\mathfrak{S}})= \Me(B_{\mathfrak{S}}, -A_{\mathfrak{S}})$. Hence 
\begin{eqnarray*}
(J\Me(A_{\mathfrak{S}}, B_{\mathfrak{S}}))^{\perp}=\Ran \begin{bmatrix} B_{\mathfrak{S}}^{\ast} \\ -A_{\mathfrak{S}}^{\ast} \end{bmatrix} &\mbox{and} & \Me(A^{\prime},B^{\prime})^{\perp}= \Ran \begin{bmatrix}(A^{\prime})^{\ast} \\ (B^{\prime})^{\ast}  \end{bmatrix}.
\end{eqnarray*}
Observe that
\begin{align*}
\left\langle  \begin{bmatrix} B_{\mathfrak{S}}^{\ast}\underline{\psi} \\ -A_{\mathfrak{S}}^{\ast} \underline{\psi} \end{bmatrix},  \begin{bmatrix}(A^{\prime})^{\ast}\underline{\varphi} \\ (B^{\prime})^{\ast} \underline{\varphi} \end{bmatrix} \right\rangle_{\Ke^2}=& \
-\frac{1}{4ik}\left\langle \underline{\psi}, \left(\mathfrak{S}^2 - \mathds{1}\right)\underline{\varphi} \right\rangle_{\Ke} + \frac{1}{4ik}\left\langle \psi, \left(\mathfrak{S}^2 - \mathds{1}\right)\varphi \right\rangle_{\Ke} \\
=&\ 0 
\end{align*}
for all $\underline{\psi},\underline{\varphi}\in \Ke$, where $\mathfrak{S}=\mathfrak{S}(k,A,B)$. Hence $(J\Me(A_{\mathfrak{S}}, B_{\mathfrak{S}}))^{\perp} \perp \Me(A^{\prime},B^{\prime})^{\perp}$, and since both spaces have dimension equal to $d$ one obtains $(J\Me(A_{\mathfrak{S}}, B_{\mathfrak{S}}))^{\perp} = \Me(A^{\prime},B^{\prime})$. Applying Proposition \ref{thm:adj} yields the claim.
\end{proof}
As a consequence, one obtains for m-sectorial operators the following
\begin{corollary}\label{Last}
Let $P$ be an orthogonal projector in $\Ke$, $P^{\perp}=\mathds{1}-P$ and $L$ and operator with $L =P^{\perp}LP^{\perp}$, then 
\begin{equation*}
\Delta(P+L,P^{\perp})^{\ast}= \Delta(P+L^{\ast},P^{\perp}).
\end{equation*}
\end{corollary}

\subsection{Approximation of boundary conditions}\label{Sec.approx}
One can ask which boundary conditions are ``close to each other'', 
and for answering this question properly 
one has to decide in which topology it is raised. 
Here, boundary conditions with the same dimension are compared to each other. 
Let us thus consider the set of subspaces 
$\Me\subset \Ke^2$ with $\dim \Me=n$; 
this is the Grassmann manifold $\Gr(2d,n)$. 
For a subspace $\Me\subset \Ke^2$ denote by $P_{\Me}$ 
the orthogonal projector in $\Ke^2$ to $\Me$. 
A metric on $\Gr(2d,n)$ is defined by
\begin{eqnarray*}
d_{n}(\Me_1,\Me_2):=\norm{P_{\Me_1}-P_{\Me_2}}.
\end{eqnarray*}
\begin{lemma}\label{prop:approx}
Let $\Me\subset\Ke^2$ with $\dim \Me=n$, and let $\Me_l\subset \Ke^2$, $l\in \N$, be a sequence of $n$-dimensional subspaces with 
$$\lim_{l\to\infty} d_n(\Me_l,\Me)=0.$$
Then the sequence of operators $-\Delta(\Me_l)$ converges in the strong graph limit to $-\Delta(\Me)$. 
\end{lemma}

\begin{proof}
Denote by $\Gamma_{\Me}\subset \He^2$ the graph of the operator $-\Delta(\Me)$ for arbitrary $\Me\subset\Ke^2$. 
In order to prove convergence of $-\Delta(\Me_l)$ to $-\Delta(\Me)$ in the strong graph limit 
(see \cite[Sec.~VIII.7, p.~293]{RS1} for the definition), 
one has to prove two items:
\begin{enumerate}
\item 
For all $(\psi_l, -\Delta(\Me_l) \psi_l)\in \He^2$, 
with $\psi_l\in \Delta(\Me_l)$ such that 
$\psi_l \rightarrow \xi$ and $-\Delta(\Me_l)\psi_l \rightarrow \eta$, 
it follows that  $(\xi,\eta) \in \Gamma_{\Me}$. 
This means $\xi \in \Dom(-\Delta(\Me))$ and $\eta=-\Delta(\Me)\xi$. 
\item 
For all $ (\psi, -\Delta(\Me) \psi) \in \Gamma_{\Me}$ there exists a sequence $\{\psi_l\}_{l\in \N}$ such that  $(\psi_l,-\Delta(\Me_l)\psi_l)\in \Gamma_{\Me_l} $ and $\psi_l \rightarrow \psi,$  $-\Delta(\Me_l)\psi_l \rightarrow -\Delta(\Me)\psi$.
\end{enumerate}

Note that $-\Delta(\Me)$ is an extension of finite rank of $-\Delta^0$ for any $\Me\subset \Ke^2$. In particular, $\De^0\subset \De$ is a closed subspace and the quotient space $\De/\De^0$ can be identified with the space of boundary values $\Ke^2$. Hence one has 
\begin{eqnarray}\label{directsum}
\De = \De^0 \dot{+} \Ke^2,
\end{eqnarray}
where $\dot{+}$ denotes the direct sum.
Let $\psi_l\to \xi$ and $-\Delta(\Me_l)\psi_l \to \eta$. Since $-\Delta$ is closed and $-\Delta(\Me_l)\psi_l=-\Delta \psi_l$ it follows that $\eta = -\xi^{\prime\prime}$. By \eqref{directsum} one has a decomposition 
\begin{eqnarray*}
\xi = \xi^0 \dot{+} [\xi] &\mbox{and} & \psi_l = \psi_l^0 \dot{+} [\psi_l] \qquad \mbox{with } \xi^0,\psi_l^0\in \De^0.
\end{eqnarray*}
By assumption one has $\psi_l \to \xi$ in the graph norm which is equivalent to the Sobolev norm in the Hilbert space $\De$. Hence, $[\psi_l]\to [\xi]$ and therefore $\xi\in \Dom(-\Delta(\Me))$ which proves (1).  

Let $\psi\in \Dom(-\Delta(\Me))$. Then by \eqref{directsum} one has the decomposition $\psi = \psi^0 \dot{+} [\psi]$. By assumption there is a sequence $[\psi_l]\to [\psi]$, and 
\begin{eqnarray*}
\psi_l = \psi^0 \dot{+} [\psi_l]\in \Dom(-\Delta(\Me_l)) & \mbox{such that } \psi_l \to \psi & \mbox{and } \psi_l^{\prime\prime} \to \psi^{\prime\prime}.
\end{eqnarray*}
This proves (2) and finishes the proof.
\end{proof}
\begin{theorem}\label{thm:conv}
Let $A,B$ define irregular boundary conditions. Then there is a sequence of regular boundary conditions $A_l,B_l$, $l\in \N$, such that $-\Delta(A_l,B_l)$ converges in the strong graph limit to $-\Delta(A,B)$.
\end{theorem}
\begin{proof}
For $\Me(A,B)\subset \Ke^2$ with $\dim\Me(A,B)=d$ one has 
by \cite[Lem.~3.2]{KPS2008}
\begin{eqnarray}\label{PM}
P_{\Me(A,B)^{\perp}}= \begin{pmatrix} A^{\ast} \\ B^{\ast} \end{pmatrix} (AA^{\ast} + BB^{\ast} )^{-1}  (A, \, B).
\end{eqnarray}
Denote by $P$ the orthogonal projector in $\Ke$ onto $\Ker B$. Then $A$ and $B_{\epsilon}$ with $$B_{\epsilon}:=B+ \epsilon P$$ 
define regular boundary conditions for $\epsilon>0$ because $B_{\epsilon}$ is invertible. Note that by \cite[Lem.~3.2]{KPS2008} 
$AA^{\ast} + B_{\epsilon}B_{\epsilon}^{\ast}$, $\epsilon\geq 0$, is invertible since $\dim\Me(A,B_{\epsilon})=d$, and therefore $$\lim_{\epsilon\to 0}(AA^{\ast} + B_{\epsilon}B_{\epsilon}^{\ast})^{-1}= (AA^{\ast} + BB^{\ast})^{-1}.$$
Using that $P_{\Me(A,B_{\epsilon})}= \mathds{1}-P_{\Me(A,B)^{\perp}}$, $\epsilon\geq 0$, and \eqref{PM} one can prove then that 
\begin{eqnarray*}
\lim_{\epsilon\to 0}\norm{P_{\Me(A,B)}-P_{\Me(A,B_{\epsilon})}}=0, 
\end{eqnarray*}
where the norm is the operator norm, but since in the finite dimensional Hilbert space $\Ke^2$ all norms are equivalent it is also sufficient to prove component wise convergence. Applying Lemma~\ref{prop:approx} to $-\Delta(A,B)$ and to $-\Delta(A,B_{1/l})$, $l\in \N$, proves the claim.
\end{proof}

One has to emphasise that convergence in the strong graph limit does not imply convergence of spectra. Consider for example the operators $-\Delta(A_{\tau},B_{\tau})$ defined in Example \ref{ex1}. These operators converge for $\tau \to \pi/2$ in the strong graph limit to the operator 
$-\Delta(A_{\pi/2},B_{\pi/2})$, which has empty resolvent set, 
whereas $-\Delta(A_{\tau},B_{\tau})$ for $\tau\neq \pi/2$ 
are similar to the self-adjoint Laplacian on the real line
with spectrum $[0,\infty)$. 
The strong graph convergence for this special example
was studied previously in \cite[Prop.~2.7]{PetrPhD}. 

The approximation of regular boundary conditions 
in the norm resolvent sense will be established
in Subsection~\ref{Sec.approx2}.

\subsection{$\mathcal{J}$-self-adjointness}
Let 
$
\mathcal{J}\colon \He \rightarrow \He
$
be an anti-linear bounded operator with bounded inverse. The operator $\Delta(\Me)$ is called 
\emph{$\mathcal{J}$-self-adjoint} if 
$$\Delta(\Me)^{\ast}= \mathcal{J}^{-1}\Delta(\Me)\mathcal{J}.$$ 
If $\mathcal{J}$ is in addition involutive and isometric, 
then our definition agrees with the standard notion from \cite[Sec.~III.5]{Evans}.
The usage of the $\mathcal{J}$-self-adjointness 
was suggested in~\cite{BK} as a generalised concept 
of $\mathcal{PT}$-symmetry. It was also pointed out there
that the residual spectrum of $\mathcal{J}$-self-adjoint operators is empty.
The latter can be easily seen as follows, also for our broader definition. 
The equality
$
  \Delta(\Me)^{\ast}-\lambda
  = \mathcal{J}^{-1}\left(\Delta(\Me)-\overline{\lambda}\right)\mathcal{J}
$ 
implies the symmetry relation
$
\overline{\sigma_\mathrm{p}(\Delta(\Me))}
=\sigma_\mathrm{p}(\Delta(\Me)^{\ast})
$. 
Using the general characterisation of the residual spectrum
\begin{eqnarray}\label{sigmar}
\sigma_\mathrm{r}(-\Delta(\Me))
=\{\lambda\notin \sigma_\mathrm{p}(-\Delta(\Me))
\mid \overline{\lambda} \in \sigma_\mathrm{p}(-\Delta(\Me)^{\ast}) 
\},
\end{eqnarray} 
it thus follows that the residual spectrum of $\Delta(\Me)$ is empty.  

Let us further assume that $\mathcal{J}$ commutes with the maximal operator $\Delta$, then $\mathcal{J}$ induces by
\begin{eqnarray*}
\underline{\mathcal{J}}\colon \Ke \rightarrow \Ke, & \underline{\psi} \mapsto \underline{\mathcal{J}\psi}
\end{eqnarray*}
an anti-linear operator in $\Ke$, because the map $\psi\mapsto \underline{\psi}$ is surjective as a map from $\We$ to $\Ke$.

\begin{proposition}\label{prop:THT}
Let $\dim \Me= d$ and $\Me=\Me(A,B)$ be such that $A-\kappa B$ is invertible for a $\kappa> 0$. Then $\Delta(\Me)^{\ast}$ is $\mathcal{J}$-self-adjoint if and only if $\mathfrak{S}(i\kappa,A,B)$ is $\underline{\mathcal{J}}$-self-adjoint, \ie,
\begin{eqnarray*}
\mathfrak{S}(i\kappa,A,B)^{\ast}=\underline{\mathcal{J}}^{-1}\,\mathfrak{S}(i\kappa,A,B)\underline{\mathcal{J}}.
\end{eqnarray*}
\end{proposition}
\begin{proof}
It is sufficient to prove that $\mathcal{J}\colon \Dom(\Delta(\Me)^{\ast}) \rightarrow  \Dom(\Delta(\Me))$ is bijective. This is equivalent to $[\mathcal{J}]\Me^{\ast}= \Me$, where $\Delta(\Me)^{\ast}=\Delta(\Me^{\ast})$ and 
\begin{eqnarray*}
[\mathcal{J}]\colon \Ke^2 \rightarrow \Ke^2, & \underline{\psi}\oplus\underline{\psi}^{\prime} \mapsto \underline{\mathcal{J}\psi}\oplus \underline{(\mathcal{J}\psi)}^{\prime}.
\end{eqnarray*}
Let $[\psi]=\underline{\psi}\oplus\underline{\psi}^{\prime}\in \Me^{\ast}$, then by Proposition~\ref{prop:adj}  
\begin{eqnarray*}
- \frac{1}{2} \left(\mathfrak{S}(i\kappa,A,B)^{\ast} -\mathds{1}\right)\underline{\psi} -\frac{1}{2\kappa} \left(\mathfrak{S}(i\kappa,A,B)^{\ast} +\mathds{1}\right)\underline{\psi}^{\prime}=0,
\end{eqnarray*}
and $[\mathcal{J}][\psi]\in \Me$ if and only if 
\begin{eqnarray*}
- \frac{1}{2} \left(\mathfrak{S}(i\kappa,A,B) -\mathds{1}\right)\underline{\mathcal{J}\psi} -\frac{1}{2\kappa} \left(\mathfrak{S}(i\kappa,A,B) +\mathds{1}\right)\underline{(\mathcal{J}\psi)}^{\prime}=0.
\end{eqnarray*}
After applying $\underline{\mathcal{J}}^{-1}$, this is equivalent to 
\begin{eqnarray*}
- \frac{1}{2} 
\left(
\underline{\mathcal{J}}^{-1} \mathfrak{S}(i\kappa,A,B)\underline{\mathcal{J}} -\mathds{1}
\right)
\underline{\psi} -\frac{1}{2\kappa} 
\left(
\underline{\mathcal{J}}^{-1}\mathfrak{S}(i\kappa,A,B)\underline{\mathcal{J}} +\mathds{1}
\right)
\underline{\psi}^{\prime}=0.
\end{eqnarray*}
Therefore $[\mathcal{J}]\Me^{\ast}=\Me$ if and only if $\mathfrak{S}(i\kappa,A,B)^{\ast}=\underline{\mathcal{J}}^{-1}\mathfrak{S}(i\kappa,A,B)\underline{\mathcal{J}}$. 
\end{proof}

An example for $\mathcal{J}$ being in addition involutive and isometric  
is the operator of complex conjugation
\begin{eqnarray*}
\mathcal{T}\colon \He \rightarrow \He, & \psi \mapsto \overline{\psi}.
\end{eqnarray*}
In quantum mechanics, $\mathcal{T}$~has the physical meaning of the time-reversion operator. We remark that the time-reversion in quantum mechanics can be more complicated in spinorial models and it can be non-involutive, \cf~for instance \cite{Kochan-2012-45}, where Pauli equation is discussed. The origin of non-involutivity is the non-trivial action of $\mathcal{J}$ on the spinor components. The similar structure of $\mathcal{J}$ can be considered in the graph case as well, \eg~the composition of $\mathcal{T}$ and permutation of edges.

\begin{corollary}
Let $\dim \Me= d$ and $\Me=\Me(A,B)$ such that $A-\kappa B$ is invertible for a $\kappa> 0$. Then $\Delta(\Me)^{\ast}$ is $\mathcal{T}$-self-adjoint if and only if
\begin{eqnarray*}
\mathfrak{S}(i\kappa,A,B)^{\ast}=\overline{\mathfrak{S}(i\kappa,A,B)}.
\end{eqnarray*}
\end{corollary}
\begin{Example}[Complex $\delta$-interactions are $\mathcal{T}$-self-adjoint]
Consider a finite metric graph and at each vertex $v\in \verts$ impose complex $\delta$-interactions with coupling constant $\gamma_{v}\in \C$. These are m-sectorial boundary conditions which can be parametrised at each vertex $v\in\verts$ by a projector $P_{v}$ and a rank one operator $L_{v}=-\frac{\gamma_{v}}{\deg(v)}P_{v}^{\perp}$, \cf~Example~\ref{Cdelta}. Hence, at each vertex
\begin{align*}
\mathfrak{S}(k,A_{v},B_{v})=& -\left(L_{v}+P_{v}+ik P^{\perp}_{v} \right)^{-1}\left(L_{v}+P_{v}-ik P^{\perp}_{v} \right)\\
=& \left(-\frac{\gamma_{v}}{\deg(v)}+ik \right)^{-1}\left(\frac{\gamma_{v}}{\deg(v)}+ik \right)P_{v}^{\perp} + P_{v}.
\end{align*} 
Consequently,
\begin{eqnarray*}
\mathfrak{S}(i\kappa,A_{v},B_{v})^{\ast}=\overline{\mathfrak{S}(i\kappa,A_{v},B_{v})},
\end{eqnarray*}
for all $\kappa>0$ such that $\frac{\gamma_{v}}{\deg(v)}+\kappa \neq 0$. Since 
\begin{eqnarray*}
\mathfrak{S}(i\kappa,A,B)= \bigoplus_{v\in V}\mathfrak{S}(i\kappa,A_{v},B_{v}),
\end{eqnarray*}
where 
\begin{eqnarray*}
A= \bigoplus_{v\in V} A_v \quad \mbox{and } B= \bigoplus_{v\in V} B_v,
\end{eqnarray*}
the operator $-\Delta(A,B)$ defined by $\delta$-interactions at each vertex is $\mathcal{T}$-self-adjoint for any complex coupling parameters, whereas it is self-adjoint only for real coupling parameters, \cf~\cite[Sec. 3.2.1]{PKQG1}. 
\end{Example}

%

\section{General spectral properties}\label{sec:spec}
%
In this section we collect some basic facts 
about the spectrum of the Laplacians on metric graphs.
  
\subsection{Non-zero eigenvalues}\label{subsec:ev}
A fundamental system of the equation 
$-\psi_k''-k^2\psi_k=0$ with $k\neq 0$
is given by the functions $e^{ik x}$ and $ e^{-ik x}$. 
For $\Im k>0$ only the first of the mentioned functions 
is square integrable on the half line $[0,\infty)$ 
and hence on the external edges. 
Consequently, 
an \textit{Ansatz} for 
an eigenfunction corresponding to 
an eigenvalue~$k^2$ satisfying $\Im k>0$ is to consider 
\begin{align*}
\psi_k(x_j)= \begin{cases} s_j(k) e^{ik x_j}, & j\in \Ee, \\
\alpha_j(k) e^{ik x_j} + \beta_j(k) e^{-ik x_j}, & j\in \Ie.  \end{cases}
\end{align*}
The function $\psi_k$ has the traces
\begin{align*}
\underline{\psi_k}= X\left( k;\underline{a} \right) \begin{bmatrix} \{s_j(k)\}_{j\in \Ee} \\ \{\alpha_j(k)\}_{j\in \Ie} \\ \{\beta_j(k)\}_{j\in \Ie} \end{bmatrix}, && 
\underline{\psi_k^\prime}= ik \cdot Y\left(k;\underline{a}\right) \begin{bmatrix} \{s_j(k)\}_{j\in \Ee} \\ \{\alpha_j(k)\}_{j\in \Ie} \\ \{\beta_j(k)\}_{j\in \Ie} \end{bmatrix},
\end{align*}
where 
\begin{eqnarray*}
X\left( k;\underline{a} \right)=\begin{bmatrix} \mathds{1} & 0 & 0 \\ 0 & \mathds{1} & \mathds{1} \\ 0 & e^{ik\underline{a}} & e^{-ik\underline{a}}\end{bmatrix} & \mbox{and}& 
Y\left( k;\underline{a} \right)=\begin{bmatrix} \mathds{1} & 0 & 0 \\ 0 & \mathds{1} & -\mathds{1} \\ 0 & -e^{ik \underline{a}} & e^{-ik\underline{a}}\end{bmatrix}
\end{eqnarray*}
are given with respect to the decomposition $\Ke=\Ke_{\Ee}  \oplus \Ke_{\Ie}^- \oplus \Ke_{\Ie}^+$. Here $e^{\pm ik\underline{a}}$ 
denote $(\abs{\Ie}\times \abs{\Ie})$-diagonal matrices with entries $\{e^{\pm ik\underline{a}}\}_{i,j}=\delta_{i,j} e^{\pm ik a_i }$. 

The function $\psi_k$ is an eigenfunction to the eigenvalue $k^2$ 
if and only if $\psi_k\in \Dom(-\Delta(A,B))$. This is the case if and only if the \textit{Ansatz} function $\psi_k$ satisfies the boundary conditions, 
which are encoded in the equation
\begin{align*}
 Z\left(k;A,B,\au\right) \begin{bmatrix} \{s_j(k)\}_{j\in \Ee} \\ \{\alpha_j(k)\}_{j\in \Ie} \\ \{\beta_j(k)\}_{j\in \Ie} \end{bmatrix}=0,
\end{align*}
where 
\begin{align*}
Z\left(k;A,B,\au\right)=AX(k;\underline{a})+ik\,BY(k;\underline{a}).
\end{align*}
Hence $k^2$ with $\Im k>0$ is an eigenvalue of $\Delta(A,B)$ 
if and only if 
\begin{align}\label{eq:detZLB}
\det Z\left(k;A,B,\au\right)=0,
\end{align}
and $k^2$ has geometric multiplicity $\dim\Ker Z\left(k;A,B,\au\right)$. 

For $\Ee=\emptyset$ the solutions of $\det Z\left(k;A,B,\au\right)=0$ for $k>0$ are also eigenvalues, whereas for $\Ie=\emptyset$ the solutions of $\det Z\left(k;A,B,\au\right)=0$ for $k>0$ are not eigenvalues. In particular, for $\Ie=\emptyset$ there are no positive real eigenvalues since neither $e^{ikx}$ nor $e^{-ikx}$ is square integrable on the half-line and therefore on the external edges. This is illustrated by the following example.
\begin{Example}[Graph with a spectral singularity]\label{ex:sing}
Consider the graph consisting of only one half-line, that is $\abs{\Ee}=1$ and $\Ie=\emptyset$, and impose the non-self-adjoint regular boundary conditions defined by $A=-i$ and $B=1$, \ie\ $-i \psi(0)+ \psi^{\prime}(0)=0$. Then $k=1$ is a solution of $\det(A+ikB)=0$, but $k^2=1$ is not an eigenvalue of $-\Delta(A,B)$. In \cite[Ex.~3]{Guseinov} it is shown that~$1$ is in the continuous spectrum, but it is a spectral singularity, 
which means that the limits 
\begin{eqnarray*}
\lim_{\epsilon\to 0+} \int_{I}
\left[
(-\Delta(A,B)-\lambda+\epsilon)^{-1} -(-\Delta(A,B)-\lambda-\epsilon)^{-1}
\right] d\lambda,
\end{eqnarray*}
where $I$ are some bounded real intervals, 
are singular in a certain sense;
see \cite[Def.~1]{Guseinov} for the precise definition 
and for further references on the topic. 
An alternative definition of spectral singularities 
is  related to the limit of the resolvent kernel 
when approaching non-isolated points in the spectrum
\cite[Def.~4]{Guseinov}. 
This phenomenon will be discussed further 
in Remark~\ref{rem:conres} below, 
after giving an explicit expression 
for the resolvent kernel in Proposition~\ref{prop:res}.
\end{Example} 
%


For self-adjoint boundary conditions it is known 
that all solutions of~\eqref{eq:detZLB}
for $k>0$ are eigenvalues \cite[Thm.~3.1]{VKRS1999}, 
including the cases $\Ee\neq \emptyset$ and $\Ie\neq \emptyset$. 
However, for non-self-adjoint boundary conditions this is not true anymore and it is difficult to study the positive real eigenvalues when $\Ee\neq \emptyset$ and $\Ie\neq \emptyset$. These eigenvalues are embedded in the essential spectrum as shown below in Subsection~\ref{subsec:ess}.

\begin{remark}\label{detZholomorphic}
The function $k\mapsto \det Z\left(k;A,B,\au\right)$ is holomorphic on the whole complex plain, hence it either vanishes identically or its zeros form a discrete set. Consequently one has for $\dim\Me(A,B)\geq d$ 
that $\clo \, \sigma_\mathrm{p}(-\Delta(A,B))$ is either entire $\C$ or at most discrete, where $\clo$ denotes the closure in $\C$.    
\end{remark}

\subsection{Eigenvalue zero}
Eigenfunctions to the eigenvalue zero are piecewise affine, because a fundamental system of the equation $\psi^{\prime\prime}=0$ is given by the constant solution and the linear solution. 
This gives the \textit{Ansatz}
\begin{eqnarray*}
\psi_0(x_j)= \begin{cases} 0, & j\in \Ee, \\
\alpha_j^0 + \beta_j^0x_j, & j\in \Ie,
\end{cases} 
\end{eqnarray*}
with traces 
\begin{align*}
\underline{\psi_0}= X_0\left(\underline{a} \right) \begin{bmatrix} 0 \\ \{\alpha_j^0\}_{j\in \Ie} \\ \{\beta_j^0\}_{j\in \Ie} \end{bmatrix}, &&
\underline{\psi_0^{\prime}}= Y_0\left(\underline{a} \right) \begin{bmatrix} 0 \\ \{\alpha_j^0\}_{j\in \Ie} \\ \{\beta_j^0\}_{j\in \Ie} \end{bmatrix},
\end{align*}
where
\begin{eqnarray*}
X_0\left(\underline{a}\right)=\begin{bmatrix} 0 & 0 &0 \\ 0 & \mathds{1} & 0 \\ 0 & \mathds{1} & \au \end{bmatrix} & \mbox{and}& 
Y_0\left(\underline{a}\right)=\begin{bmatrix} 0 & 0 &0 \\ 0 & 0 & \mathds{1} \\ 0 & 0 & -\mathds{1}  \end{bmatrix}.
\end{eqnarray*}
Consequently zero is an eigenvalue of the operator $-\Delta(A,B)$ if and only if there are $\alpha_j^0$, $\beta_j^0$, with $j\in \Ie$, such that  
\begin{align*}
\left[A X_0\left(\underline{a} \right) + B Y_0\left(\underline{a} \right) \right]\begin{bmatrix} 0 \\ \{\alpha_j^0\}_{j\in \Ie} \\ \{\beta_j^0\}_{j\in \Ie} \end{bmatrix}=0
\end{align*}
has a non-trivial solution. For $\Ee=\emptyset$ zero is an eigenvalue if and only if
\begin{align*}
\det \left(AX_0(\underline{a}) + BY_0(\underline{a})\right)=0,
\end{align*}
and for $\Ie=\emptyset$ zero cannot be an eigenvalue.

\subsection{Operators with empty resolvent set}
For non-self-adjoint Laplacians the resolvent set is not always non-empty, 
and one needs a certain number of boundary conditions 
to define operators of which the spectrum forms a proper subset of~$\C$.

\begin{proposition}\label{prop:dimM}
Let $\dim\Me\neq d$, then $\sigma(\Delta(\Me))=\C$. 
In particular, if $\dim\Me >d$ 
then $\clo\, \sigma_\mathrm{p}(\Delta(\Me))=\C$, 
where $\clo$ denotes the closure in $\C$.
\end{proposition}
\begin{proof}
For $\Me$ with $\dim\Me >d$ there are maps $A,B$ in $\Ke$ such that $\Me=\Me(A,B)$. By assumption the map $(A,\, B)$ is not surjective, and therefore also 
\begin{equation*}
Z\left(k;\underline{a},A,B\right)= (A,\, B) \circ \begin{pmatrix} X(k;\au) \\ ik\,   Y(k;\au)\end{pmatrix}
\end{equation*}
is not surjective for any $k$. Consequently $\det Z\left(k;\underline{a},A,B\right)=0$ for all $k\in \C$ which proves that $\C\setminus[0,\infty)\subset\sigma_\mathrm{p}(-\Delta(A,B))$. 
Since the spectrum is a closed set,
we conclude with $\sigma(\Delta(A,B))=\C$.

Let $\dim \Me<d$. 
Then, by Corollary~\ref{cor:2d}, 
$\dim \Me^{\ast}>d$, and hence $\sigma(\Delta(\Me)^{\ast})=\C$. 
Since $\sigma(\Delta(\Me))= \overline{\sigma(\Delta(\Me)^{\ast})}$, \cf~\cite[Thm. III.6.22]{Kato}, 
the claim follows.
\end{proof}
As already discussed for irregular boundary conditions 
defined by $A,B$, the resolvent set can be empty 
even if $\dim \Me(A,B)=d$,
\cf~Example~\ref{sgnsgn}.

\subsection{Residual spectrum for regular boundary conditions}
Following \cite[Eq.~(3.7)]{VKRS2006},
for regular boundary conditions with $A+ik B$ invertible 
the secular equation \eqref{eq:detZLB} can be rewritten using the identity
\begin{eqnarray}\label{1ST}
Z\left(k;A,B,\au\right)= (A+ik B)\left[ \mathds{1} - \mathfrak{S}(k,A,B)T(k;\au) \right]R_+(k;\au),
\end{eqnarray}
where
\begin{eqnarray*}
T(k;\au)= \begin{bmatrix} 0 & 0 & 0 \\ 0 & 0 & e^{ik\au} \\0 & e^{ik\au} & 0  \end{bmatrix} &\mbox{and} & R_+(k,\au)=\begin{bmatrix} \mathds{1} & 0 & 0 \\ 0 & \mathds{1} & 0 \\0 & 0 & e^{-ik\au}  \end{bmatrix}.
\end{eqnarray*}
In particular one obtains 
\begin{lemma}\label{lemma:spec}
Let $A,B$ define regular boundary conditions. 
Then $$\clo \sigma_\mathrm{p}(-\Delta(A,B))\neq \C,$$
and $\lambda\in\sigma_\mathrm{p}(-\Delta(A,B))\setminus [0,\infty)$ 
if and only if 
$\overline{\lambda}\in\sigma_\mathrm{p}(-\Delta(A,B)^{\ast})
\setminus [0,\infty)$. 
\end{lemma}
\begin{proof}
For $A,B$ defining regular boundary conditions $\mathfrak{S}(i\kappa,A,B)$ is defined for every $\kappa>0$ except a finite set and $\mathds{1} - \mathfrak{S}(i\kappa,A,B)T(i\kappa;\au)$ is invertible for $\kappa$ large enough and the first claim follows from \eqref{1ST}.

To prove the second claim,
we first notice that 
for $A\pm ik B$ invertible one has with $\eqref{1ST}$
\begin{multline*}
 Z\left(k;A,B,\au\right) \\
 = (A+ik B) \left[ \mathds{1} - \mathfrak{S}(k,A,B)T(k;\au) \right]\mathfrak{S}(k,A,B)\mathfrak{S}(k,A,B)^{-1} R_+(k;\au).
\end{multline*}
Since 
$Z\left(k;A,B,\au\right)$ is holomorphic in entire $\C$ the above representation admits continuous continuation to $\C$. So, taking the adjoint one obtains
\begin{eqnarray*}
\overline{\det Z\left(k;A,B,\au\right)}= \det(A^{\ast}-i\overline{k} B^{\ast}) \det\left[ \mathds{1} -\mathfrak{S}(k,A,B)^{\ast} T(-\overline{k};\au) \right]\det R_+(-\overline{k};\au)
\end{eqnarray*}
for all $k\in \C$ except a finite set, where one has used 
\begin{align*}
&\det\left(\left[\mathfrak{S}(k,A,B)^{\ast}\right]^{-1}\mathfrak{S}(k,A,B)^{\ast}\left[ \mathds{1} - T(k;\au)^{\ast}\mathfrak{S}(k,A,B)^{\ast} \right]\right)\\
=&\det\left[ \mathds{1} -\mathfrak{S}(k,A,B)^{\ast} T(-\overline{k};\au) \right].
\end{align*}
Applying Proposition \ref{prop:adj} 
and choosing the representation $A_{\mathfrak{S}},B_{\mathfrak{S}}$ 
given there, we arrive at
\begin{eqnarray*}
\mathfrak{S}(k,A,B)^{\ast}= \mathfrak{S}(-\overline{k},A^{\prime},B^{\prime})
\end{eqnarray*}
and hence 
\begin{eqnarray*}
\overline{\det Z\left(k;A,B,\au\right)}=\det Z\left(-\overline{k};A^{\prime},B^{\prime},\au\right) 
\end{eqnarray*}
for all $k\in \C$ except a finite set. By continuous continuation the claim follows for all $k\in \C$, and hence $k^2$, $\Im k>0$, is an eigenvalue of $-\Delta(A,B)$ if and only if $\overline{k}^2$ is an eigenvalue of $-\Delta(A,B)^{\ast}$.    
\end{proof}
\begin{remark}
Note that for $\Ee=\emptyset$ one can even show that 
$\lambda\in\sigma_\mathrm{p}(-\Delta(A,B))$ 
if and only if $\overline{\lambda}\in\sigma_\mathrm{p}(-\Delta(A,B)^{\ast})$.
\end{remark}
\begin{remark}[Stability of eigenvalues 
under similarity of scattering matrices]
Let $(\Ge,\au)$ be a compact finite metric graphs. Let $A,B$ and $A^{\prime},B^{\prime}$ define regular boundary conditions, and assume that there is an invertible map $G(k)$, $k\in \C$, such that
\begin{eqnarray*}
\mathfrak{S}(k,A,B)= G(k)^{-1}\mathfrak{S}(k,A^{\prime},B^{\prime})G(k) & \mbox{and} & G(k)T(k;\au)=T(k;\au)G(k), 
\end{eqnarray*} 
for all $k\in\C$. Then using \eqref{1ST} and one obtains immediately 
$$
  \sigma_\mathrm{p}(-\Delta(A,B))
  =\sigma_\mathrm{p}(-\Delta(A^{\prime},B^{\prime}))
$$ 
and the geometric multiplicity of the eigenvalues agrees.
\end{remark}

Combining Lemma~\ref{lemma:spec} with
the general characterisation of 
the residual spectrum \eqref{sigmar}, 
we obtain the following useful property.
\begin{proposition}
Let $A,B$ define regular boundary conditions 
then the residual spectrum $\sigma_\mathrm{r}(-\Delta(A,B))$ 
is contained in $[0,\infty)$. 
If $\Ee=\emptyset$ or $\Ie=\emptyset$ then 
$\sigma_\mathrm{r}(-\Delta(A,B))=\emptyset$.
\end{proposition}
In particular, using \eqref{sigmar} and Remark~\ref{detZholomorphic} it follows that the residual spectrum forms a discrete subset of $[0,\infty)$. 
That the residual spectrum is in general not empty 
is shown by the following example.
\begin{Example}[Graph with a residual spectrum]
Consider the metric graph consisting of one internal edge of length $a$ and one external edge. Impose the following boundary conditions
\begin{eqnarray*}
\psi_{\Ee}^{\prime}(0)=0, &  -i \psi_{\Ie}(0)+ \psi^{\prime}_{\Ie}(0)=0, & \psi_{\Ee}(0) + i \psi_{\Ie}(a)-\psi^{\prime}_{\Ie}(a)=0.
\end{eqnarray*}
These are $m$-sectorial boundary conditions with 
\begin{eqnarray*}
P=0 & \mbox{and} & L= \begin{bmatrix} 0 & 0 & 0 \\ 0 & -i & 0 \\ 1 & 0 & i \end{bmatrix}.
\end{eqnarray*}
A direct computation shows that 
\begin{equation*}
\psi(x) = \begin{cases} 0, & x\in \Ee, \\ e^{ix}, & x\in \Ie, \end{cases}
\end{equation*}
is an eigenfunction of $-\Delta(L,\mathds{1})$ 
corresponding to the eigenvalue~$1$. 
By Corollary~\ref{Last} the adjoint operator is given by  
$-\Delta(L^*,\mathds{1})$, which is defined by the boundary conditions  
\begin{eqnarray*}
\psi_{\Ie}(a)  + \psi_{\Ee}^{\prime}(0)=0, &  i \psi_{\Ie}(0)+ \psi^{\prime}_{\Ie}(0)=0, & i \psi_{\Ie}(a)+\psi^{\prime}_{\Ie}(a)=0.
\end{eqnarray*}
For the second condition an eigenfunction 
corresponding to the eigenvalue~$1$ would be $e^{-ix}$ 
on the internal edge and for the square integrability~$0$ 
on the external edge, but this function does not satisfy the third boundary condition nor the first. Therefore $1$ is an eigenvalue of $-\Delta(L,\mathds{1})$, but not an eigenvalue of $-\Delta(L^{\ast},\mathds{1})$. 
Using the characterisation of the residual spectrum in \eqref{sigmar}, 
one obtains that $1\in \sigma_\mathrm{r}(-\Delta(L^*,\mathds{1}))$. 
\end{Example}

\subsection{Resolvents for regular boundary conditions}\label{subsec:res}
In {\cite[Lem.~3.10]{KPS2008}} 
an explicit formula for the resolvent 
associated with $k^2\in \rho(-\Delta(A,B))$ is given. 
In this subsection we reproduce the result 
for regular boundary conditions 
and add a criterion for $k^2$ being in the resolvent set. 
Since the result of \cite[Lem.~3.10]{KPS2008} is given without proof
(arguing that it can be proved ``in the exactly the same way''
as for self-adjoint boundary conditions),
we provide a short proof in the appendix
(where we also recall the notion of integral operators,
\cf~Definition~\ref{def:integralop}).
This will make our paper self-consistent 
and, moreover, 
clarify the need for regularity of boundary conditions in the proof.

\begin{proposition}\label{prop:res}
Let $A,B$ define regular boundary conditions such that 
\begin{eqnarray*}
A\pm ik B & \mbox{and} & \mathds{1} - \mathfrak{S}(k,A,B)T(k;\au)
\end{eqnarray*}
are invertible for $k\in\C$ with $\Im k>0$. Then $k^2\in \rho(-\Delta(A,B))$ and the resolvent $\left( -\Delta(A,B)-k^2\right)^{-1}$ is the integral operator with the $(\abs{\Ee}+ \abs{\Ie})\times (\abs{\Ee}+ \abs{\Ie})$ matrix valued integral kernel $r_{\Me}(x,y;k)$, $\Me=\Me(A,B)$, admitting the representation 
\begin{align*}
&r_{\Me}(x,y;k)= r^0(x,y;k) + r^1_{\Me}(x,y;k)
\end{align*}
with 
$
\{r^0(x,y;k)\}_{j,j^{\prime}}
= \delta_{j,j^{\prime}}\frac{i}{2k} e^{ik\abs{x_j-y_j}}
$
and
\begin{multline*}
  r_{\Me}^1(x,y;k)\\ 
  =\frac{i}{2k} \Phi(x,k) R_+(k;\au)^{-1}\left[\mathds{1}- 
  \mathfrak{S}(k,A,B) T(k;\au)\right]^{-1}\mathfrak{S}(k,A,B) 
  R_+(k;\au)^{-1}\Phi(y,k)^T,
\end{multline*}
where
the matrix $\Phi(x,k)$ is given by 
\begin{equation*}
\Phi(x,k):= \begin{bmatrix} \phi(x,k) & 0 & 0 \\ 0 & \phi_+(x,k)& \phi_-(x,k) \end{bmatrix}
\end{equation*}
with diagonal matrices $\phi(x,k)=\diag\{ e^{ikx_j}\}_{j \in \Ee}$ and $\phi_{\pm}(x,k)=\diag\{ e^{\pm ikx_j}\}_{j\in \Ie}$, and $\Phi(x,k)^T$ denotes the transposed of $\Phi(x,k)$.
\end{proposition}

\begin{remark}
The statement of the proposition holds also for $k>0$ if $\Ee=\emptyset$.
\end{remark}
\begin{remark}\label{rem:conres}
Note that the resolvent kernel $r_{\Me}(x,y,k)$ is still well-defined for $k>0$ such that 
\begin{eqnarray*}
A\pm ik B & \mbox{and} & \mathds{1} - \mathfrak{S}(k,A,B)T(k;\au)
\end{eqnarray*}
are invertible. For these $k>0$ the kernel $r_{\Me}(x,y,k)$ still defines an operator from $L^2(\Ge, e^{x\delta}dx)$ to $L^2(\Ge, dx)$ for $\delta>0$. In the sense of \cite[Def.~4]{Guseinov} the points $k^2>0$ such that $\lim_{\epsilon \to 0+}r_{\Me}(x,y,k+i\epsilon)$, $k>0,$ is unbounded are called spectral singularities. Example \ref{ex:sing} shows that the spectral singularities can form a larger set than the set of embedded eigenvalues. 
\end{remark}

\subsection{Approximation of regular boundary conditions}\label{Sec.approx2}
Using the explicit formula for the resolvent, 
one can establish a norm resolvent convergence 
for certain regular boundary conditions.
\begin{proposition}\label{prop:resconv}
Let $A_{\epsilon},B_{\epsilon}$, $\epsilon\geq 0$, define regular boundary conditions such that 
\begin{eqnarray*}
A_{\epsilon}\pm ik B_{\epsilon} & \mbox{and} & \mathds{1} - \mathfrak{S}(k,A_{\epsilon},B_{\epsilon})T(k;\au)
\end{eqnarray*}
are invertible for a certain $k\in\C$ with $\Im k>0$ and all $\epsilon\geq 0$. Assume furthermore that 
\begin{eqnarray*}
\lim_{\epsilon\to 0} \mathfrak{S}(k,A_{\epsilon},B_{\epsilon})= \mathfrak{S}(k,A_{0},B_{0}).
\end{eqnarray*}
Then $k^2\in \rho(-\Delta(A_{\epsilon},B_{\epsilon}))$ for all $\epsilon \geq 0$, and 
\begin{eqnarray*}
\lim_{\epsilon \to 0}\left\Vert\left( -\Delta(A_{\epsilon},B_{\epsilon})-k^2\right)^{-1}-\left( -\Delta(A_{0},B_{0})-k^2\right)^{-1}\right\Vert=0.
\end{eqnarray*}
\end{proposition}
\begin{proof}
Set $\Me_{\epsilon}:=\Me(A_{\epsilon}, B_{\epsilon})$ for $\epsilon\geq 0$. Then 
\begin{eqnarray*}
r_{\Me_{\epsilon}}(x,y;k)- r_{\Me_{0}}(x,y;k)= r_{\Me_{\epsilon}}^1(x,y;k)- r_{\Me_{0}}^1(x,y;k).
\end{eqnarray*}
Note that $r_{\Me_{\epsilon}}^1(\cdot,\cdot;k)$ define for every $\epsilon\geq 0$ Hilbert-Schmidt operators. One obtains 
\begin{multline*}
\norm{r_{\Me_{\epsilon}}^1(\cdot,\cdot;k)- 
r_{\Me_{0}}^1(\cdot,\cdot;k)}_\mathrm{HS} 
\\
\leq \frac{C(k)}{2k}  
\Big\Vert \left[\mathds{1}- \mathfrak{S}(k,A_{\epsilon},B_{\epsilon}) T(k;\au)\right]^{-1}\mathfrak{S}(k,A_{\epsilon},B_{\epsilon}) 
\\
- \left[\mathds{1}- \mathfrak{S}(k,A_{0},B_{0}) T(k;\au)\right]^{-1}\mathfrak{S}(k,A_{0},B_{0}) \Big\Vert ,
\end{multline*}
because $r(x,y;k)=\Phi(x,k) R_+(k;\au)^{-2}\Phi(y,k)^T$ 
defines a Hilbert-Schmidt operator $R(k)$, 
with a finite Hilbert-Schmidt norm 
$\norm{r(\cdot,\cdot;k)}_\mathrm{HS} =: C(k)$. 
From the convergence of $\mathfrak{S}(k,A_{\epsilon},B_{\epsilon})$ to $\mathfrak{S}(k,A_{0},B_{0})$ it follows under the assumptions imposed that  
$$\lim_{\epsilon\to 0}\left[\mathds{1}- \mathfrak{S}(k,A_{\epsilon},B_{\epsilon}) T(k;\au)\right]^{-1}= \left[\mathds{1}- \mathfrak{S}(k,A_{0},B_{0}) T(k;\au)\right]^{-1}.$$
Hence,  
\begin{align*}
0\leq&\lim_{\epsilon \to 0} \left\Vert\left( -\Delta(A_{\epsilon},B_{\epsilon})-k^2\right)^{-1}-\left( -\Delta(A_{0},B_{0})-k^2\right)^{-1}\right\Vert \\
\leq &\lim_{\epsilon \to 0}  \left\Vert\left( -\Delta(A_{\epsilon},B_{\epsilon})-k^2\right)^{-1}-\left( -\Delta(A_{0},B_{0})-k^2\right)^{-1}
\right\Vert_\mathrm{HS}=0 
,
\end{align*}
which proves the claim.
\end{proof}

In contrast to the convergence in the strong graph sense 
established in Subsection~\ref{Sec.approx},
the norm resolvent convergence implies
the convergence of spectra.

\subsection{Essential spectra 
for regular boundary conditions}\label{subsec:ess}
For non-self-adjoint operators 
there are various notions of the essential spectrum. 
Five types, defined in terms of Fredholm properties
and denoted by $\sigma_{\mathrm{e}j}$ for $j=1,2,3,4,5$,
are in detail discussed in \cite[Chap.~IX]{Evans}.
All these essential spectra coincide for~$T$ self-adjoint, 
but for closed non-self-adjoint~$T$ 
one has in general only the inclusions 
$\sigma_{\mathrm{e}j}(T)\subset\sigma_{\mathrm{e}_i}(T)$ with $j<i$.  
The largest set $\sigma_{\mathrm{e}5}(T)$ is known as 
the essential spectrum due to Browder and 
it coincides with the complement in the spectrum 
of isolated eigenvalues~$\lambda$
of finite algebraic multiplicity such that $\Ran(T-\lambda)$ is closed.
\begin{proposition}
Let $A,B$ define through~\eqref{AB} regular boundary conditions. 
Then $\rho(-\Delta(A,B))\neq \emptyset$. 
For $\Ee\neq \emptyset$ one has 
$\sigma_{\mathrm{e}5}(-\Delta(A,B))= [0,\infty)$. 
For $\Ee=\emptyset$ the spectrum is purely discrete 
and the resolvent is compact, hence 
$\sigma_{\mathrm{e}5}(-\Delta(A,B))= \emptyset$. 
\end{proposition}
\begin{proof}
From Lemma~\ref{lemma:spec} together with Proposition~\ref{prop:res} it follows that the resolvent set is not empty and that the resolvents for any regular boundary conditions differ only by a perturbation of finite rank.
In particular, the difference of respective resolvents is compact. 

Assume $\Ee\neq \emptyset$.
Self-adjoint realisations are also defined by regular boundary conditions
and it is well known that the essential spectrum is $[0,\infty)$
in this case.
Applying the Weyl-type perturbation result from \cite[Thm.~IX.2.4]{Evans},
it follows that  
$\sigma_{\mathrm{e}i}(-\Delta(A,B))= [0,\infty)$ 
with $i=1,2,3,4$.
Since $\C\setminus\sigma_{\mathrm{e}1}(-\Delta(A,B))$ 
has only one connected component, 
which intersects the resolvent set of $-\Delta(A,B)$,
$\sigma_{\mathrm{e}5}(-\Delta(A,B))=\sigma_{\mathrm{e}1}(-\Delta(A,B))$
by the very definition of~\cite[Chap.~IX]{Evans}.
 
Now let $\Ee=\emptyset$. 
Then all self-adjoint realisations have compact resolvent, 
Proposition~\ref{prop:res} applies 
and the resolvents for any regular boundary conditions 
differ only by a perturbation of finite rank. 
Hence the resolvent is compact for all regular 
boundary conditions which proves the assertion.  
\end{proof}

In particular, one obtains that on finite compact metric graphs for regular boundary conditions the spectrum is purely discrete, there is no continuous nor residual spectrum. 
For graphs with $\Ie=\emptyset$ there are at most finitely many eigenvalues
in $\C\setminus[0,\infty)$,
they have finite algebraic multiplicity, 
and the continuous spectrum is $[0,\infty)$, 
whereas the residual spectrum is empty. 
For the case $\Ee\neq \emptyset$ and $\Ie\neq \emptyset$ it is difficult to give general statements on eigenvalues and residual spectrum contained in $[0,\infty)$.

\section{Riesz basis on compact graphs}\label{sec:riesz} 
%
In this section we apply 
a general result due to Agranovich~\cite{Agr94}
about a Riesz basis property to m-sectorial Laplacians
on finite compact metric graphs.
Throughout this section let us therefore assume 
that $(\Ge,\au)$ is an arbitrary finite compact metric graph, 
\ie~$\Ee=\emptyset$. 

Let us first recall the definition of the Riesz basis of subspaces; 
see, \eg, \cite{Markus} for more details. 
The set of subspaces 
$\{\mathcal{N}_k\}_{k=1}^{\infty} \subset \cH$ 
is called a \emph{basis of subspaces} 
if any vector $f$ from the Hilbert space~$\cH$ 
can be uniquely represented as a series
$$
f=\sum_{k=1}^{\infty} f_k, \qquad f_k \in \mathcal{N}_k. 
$$
Such a basis is called unconditional or \emph{Riesz} 
if it remains a basis after any permutation of 
the subspaces appearing in it, \ie, 
if the above series converges unconditionally for any~$f$. 
If the subspaces are one dimensional, 
we obtain the standard notion of Riesz basis.

Let $P$ be an orthogonal projector in $\Ke$, 
$P^{\perp}=\mathds{1}-P$ its complementary projector 
and $L$~a not necessarily self-adjoint operator in $\Ke$ 
with $L=P^{\perp}LP^{\perp}$. 
Then one considers $-\Delta(P+L,P^{\perp})$. Recall that this operator is associated with the closed sectorial form $\delta_{P,L}$ 
defined by~\eqref{deltaPL}. 
The main result of this section reads as follows.
\begin{theorem}\label{thm:riesz}
The spectrum of the operator $-\Delta(P+L,P^{\perp})$ is purely discrete, 
and there is a Riesz basis consisting of finite dimensional 
invariant subspaces of $-\Delta(P+L,P^{\perp})$.
\end{theorem}

The proof of Theorem~\ref{thm:riesz} is based
on the following abstract result due to Agranovich~\cite{Agr94}.
\begin{theorem}[{\cite[Thm.~in Sec.~1]{Agr94}}]\label{thm:Agr}
Let $\He$ and $\We\subset \He$ be separable Hilbert spaces, 
where the imbedding $\We \hookrightarrow \He$ is compact. 
Consider a closed sectorial form $\mathfrak{a}$ 
with domain $\Dom \mathfrak{a}=\We$, and
denote by $A$ the m-sectorial operator defined by $\mathfrak{a}$. 
Assume that there are constants $c,C>0$ such that
\begin{eqnarray}\label{cond1}
c \, \norm{\psi}_{\We}^2 \leq \Re \mathfrak{a}[\psi] 
& \mbox{for all $\psi\in\We$}
\end{eqnarray}
and 
\begin{eqnarray}\label{cond2}
\abs{\mathfrak{a}[\psi,\varphi]}+ \abs{(\Re \mathfrak{a})[\psi,\varphi]} \leq C \norm{\psi}_{\We}\norm{\varphi}_{\We}  & \mbox{for all $\psi,\varphi\in\We$},
\end{eqnarray}
where $\Re \mathfrak{a}$ denotes the real part of the form $\mathfrak{a}$.
Let~$B$ be the operator defined by $\Re \mathfrak{a}$ 
and assume furthermore that for some $0\leq q <1$ and $\gamma>0$ 
\begin{eqnarray}\label{cond3}
\abs{\Im \mathfrak{a}[\psi]} \leq \gamma \norm{B^{1/2}\psi}^{2q} \norm{\psi}^{2-2q} & \mbox{for all } \psi\in\We,
\end{eqnarray}
where $\Im \mathfrak{a}$ denotes the imaginary part of~$\mathfrak{a}$.
Denote by $\lambda_1\leq \lambda_2 \leq \ldots \leq \lambda_j\leq \ldots$ the eigenvalues of $B$ (counting multiplicities) and assume that for some $p>0$
\begin{eqnarray}\label{cond4}
\limsup_{j\to\infty} \lambda_j j^{-p}>0.
\end{eqnarray}
Then there exists a Bari basis if $p(1-q)>1$, a Riesz basis if $p(1-q)=1$, and an Abel basis of order $\beta=\beta_0+\beta_1$ if $p(1-q)<1$, consisting of finite dimensional subspaces invariant with respect to $A$ respectively. Here, $\beta_0=p^{-1}-(1-q)$ and $\beta_1$ is an arbitrarily small positive number. 
\end{theorem}

To apply Theorem~\ref{thm:Agr}, 
we need the following elementary inequality, 
which we state here without proof.
\begin{lemma}\label{gn}
There exists a constant $C>0$ such that for all $\psi\in W^{1,2}((0,a))$ 
\begin{eqnarray*}
\norm{\psi}_{L^\infty}^2\leq C \norm{\psi}_{\We} \norm{\psi}.
\end{eqnarray*} 
\end{lemma}

Now we are in a position to prove Theorem~\ref{thm:riesz}.

\begin{proof}[Proof of Theorem~\ref{thm:riesz}]
Consider the form $\mathfrak{a}^{\prime}:=\delta_{P,L}$ defined by \eqref{deltaPL}. We apply Theorem~\ref{thm:Agr} to the form $\mathfrak{a}:=\mathfrak{a}^{\prime}+\epsilon$ with an appropriate $\epsilon>0$. 

There is an $\epsilon >0$ such that
the form $\mathfrak{b}:=\Re \mathfrak{a}+\epsilon> 0$ 
defines a norm that is equivalent to the Sobolev norm of~$\We$.
Indeed, using Lemma~\ref{gn} together with Young's inequality,
we have
\begin{align*}
-\langle \Re L\underline{\psi},\underline{\psi}\rangle_{\Ke} &\geq - \norm{\Re L} \, \norm{\underline{\psi}}^2 
\geq 
-C \|\psi\|_{\We} \|\psi\| \\
&\geq - \delta \norm{\psi}_{\We}^2 - \frac{C^2}{4\delta} \norm{\psi}^2, 
\qquad\mbox{for any } \delta>0.
\end{align*}
Hence, for $\delta<1$ and $\epsilon>0$ such that $\frac{C^2}{4\delta} < \epsilon$, 
one has 
\begin{align*}
\int_{\Ge}\abs{\psi^{\prime}}^2 -\langle \Re L\underline{\psi},\underline{\psi} \rangle + \epsilon \norm{\psi}^2 \geq \gamma \norm{\psi}_{\We}^2, 
\end{align*}
where
\begin{eqnarray*}
\gamma=\min \left\{\epsilon-\frac{C^2}{4\delta}, 1- \delta \right\}.
\end{eqnarray*}
The other inequality can be shown analogously.

Using the notation of Theorem~\ref{thm:Agr}, we have 
\begin{align*}
\Re \mathfrak{a}[\psi]
&=\int_{\Ge} \abs{ \psi^{\prime}}^2 - \langle \Re L P^{\perp}\underline{\psi},P^{\perp}\underline{\psi}\rangle_{\Ke} + \epsilon \norm{\psi}^2,
\\
\Im \mathfrak{a}[\psi]
&= -\langle \Im L P^{\perp}\underline{\psi},P^{\perp}\underline{\psi}\rangle_{\Ke}, 
\end{align*}
where 
\begin{eqnarray*}
\Dom \mathfrak{a}
=\Dom \Re\mathfrak{a} 
= \We_P :=\{\psi \in \We \mid P\underline{\psi}=0\}\subset \He.
\end{eqnarray*}
The space $\We$ with the inner product $\langle \cdot,\cdot \rangle_{\We}$ 
is a Hilbert space and $\We_P$ is a closed subspace. 
Since $(\Ge,\au)$ is compact, 
$\We_P$~is compactly embedded in $\He$. 
Condition~\eqref{cond1} is fulfilled for $\Re\mathfrak{a}$ and Condition~\eqref{cond2} follows as well by applying the Cauchy-Schwarz inequality. Recall that the norm defined by $\mathfrak{b}$ is equivalent to the Sobolev norm in the Hilbert space $\We_P$. Therefore, there is a constant $C>0$ such that $\norm{\psi^{\prime}}\leq\norm{\psi}_{\We}\leq C \norm{B^{1/2}\psi}=C\mathfrak{b}[\psi]$. Applying Lemma~\ref{gn} to the form $\Im\mathfrak{a}$ yields   
\begin{eqnarray*}
\abs{\Im\mathfrak{a}[\psi]}\leq \norm{\Im L}\, \norm{\underline{\psi}}_{\Ke}^2 \leq C \norm{\psi}\norm{\psi}_{\We}\leq C \norm{\psi} \norm{B^{1/2}\psi},
\end{eqnarray*}
where $C>0$ is used as universal constant. Thus Condition~\eqref{cond3} is fulfilled with $q=1/2$. 

From \cite[Thm.~3.1]{AccQG} it follows that the operator associated with $\mathfrak{b}$ is the self-adjoint operator $B=-\Delta(P+\Re L,P^{\perp})+ \epsilon$. Since its spectrum is discrete, there is a variational characterisation of the eigenvalues in terms of the minimax principle. Applying a Dirichlet-Neumann-bracketing one arrives at the conclusion that $\lambda_j=O(j^2)$, see, \eg, \cite[Prop.~4.2]{Bolte2009}, and hence Condition \eqref{cond4} holds for $p=2$. 

Putting the pieces together, we obtain that Theorem~\ref{thm:Agr} applies to $\mathfrak{a}$, which defines the operator $A=A^{\prime}+\epsilon$, where $A^{\prime}=-\Delta(P+L,P^{\perp})$. Since the invariant subspaces of $A$ and $A^{\prime}$ agree, and furthermore $p(1-q)=1$ holds, these form a Riesz basis. 
\end{proof}

Theorem \ref{thm:riesz} can be applied to the following example and its below mentioned generalisations.
\begin{Example}[Complex Robin boundary conditions]
Consider the interval $[0,a]$ and impose the boundary condition 
\begin{eqnarray*}
\psi^{\prime}(0)+ (i\alpha-\beta) \psi(0) & \mbox{and} & \psi^{\prime}(a)+ (i\alpha+\beta) \psi(a)=0, \quad \mbox{for } \alpha,\beta\in \R,
\end{eqnarray*}
\cf~\cite[Sec.~6.3]{KBZ}.
In matrix notation this becomes 
\begin{eqnarray*}
A=\begin{bmatrix} i\alpha-\beta & 0 \\ 0 & -(i\alpha+\beta)  \end{bmatrix} & \mbox{and} & B=\begin{bmatrix} 1 & 0 \\ 0 & 1  \end{bmatrix},
\end{eqnarray*}
hence one has also a parametrisation in terms of m-sectorial boundary conditions with  $L=A$ and $P=0$. Therefore, the operator $-\Delta(A,B)$ is defined by the form $\delta_{L,0}$ and Theorem \ref{thm:riesz} applies. 

For $\beta=0$ an explicit computation shows that the spectrum is real.
Moreover, if $\alpha\neq n\pi/ a$, $n\in\N$,
all the eigenvalues have algebraic multiplicity one.
We refer to~\cite{KBZ} for more details. 

In fact, it is well-known that the eigensystem $-\Delta(A,B)$ contains a Riesz basis (without brackets), \cite{Mikhajlov-1962-3}, \cite[Sec.XIX.3]{DSIII}. 
These boundary conditions were introduced in~\cite{KBZ}
as a toy quasi-self-adjoint model in $\mathcal{PT}$-symmetry and the closed formula
for the metric operator has been found.
An alternative method how to derive other closed formulae 
for the metric operators~$\Theta$ was developed in~\cite{K}
and further employed in~\cite{KSZ},
where one can additionally find an explicit formula for
the self-adjoint operator to which $-\Delta(A,B)$ is similar.
Notice that this self-adjoint operator is not a graph.
The more general model with $\beta\not=0$
is also studied in~\cite{KS}. 

A generalisation of this example to metric graphs was proposed 
in \cite{Znojil}. 
Consider a compact star graph and the boundary conditions 
\begin{eqnarray*}
A=\begin{bmatrix} A_+(\alpha) & 0  \\ 0 & A_-  \end{bmatrix} & \mbox{and} & B=\begin{bmatrix} \mathds{1} & 0 \\ 0 & B_-  \end{bmatrix},
\end{eqnarray*}
where $\{A_+(\alpha)\}_{lk}=i\alpha\delta_{lk} e^{2 \pi  i \frac{ l}{\deg{v}}}$, $A_-=A_{\nu}$ and $B_-=B_{\nu}$ are the matrices given in \eqref{bc:st} below defining the standard boundary conditions at the central vertex $v$ with $\deg (v)=\nu$. Since the standard boundary conditions can be expressed equivalently by projectors $P_{\nu}$ and $P_{\nu}^{\perp}$, \cf~Subsection~\ref{subsec:star}, one has that  
\begin{eqnarray*}
L(\alpha)= \begin{bmatrix} A_+(\alpha) & 0 \\ 0 & 0  \end{bmatrix}, &\mbox{and} & P=\begin{bmatrix} 0 & 0 \\ 0 & P_{\nu}  \end{bmatrix}.
\end{eqnarray*}
Hence, Theorem~\ref{thm:riesz} applies and there is a Riesz basis of projectors corresponding to invariant subspaces of $-\Delta(A,B)$.
\end{Example}

Theorem \ref{thm:riesz} can be also applied to a compact graph with the combination of self-adjoint boundary conditions and complex $\delta$-interactions, i.e. a modification of Example \ref{Cdelta}.

\section{Quasi-self-adjointness for symmetric graphs}\label{sec:sim}
%
There are many works dealing with the question of similarity 
between non-self-adjoint and self-adjoint operators. 
In particular, there exists an abstract resolvent criterion
for similarity to self-adjoint operators developed independently in 
\cite{Casteren}, \cite{Malamud} and \cite{Naboko}. 
Based on this criteria, the question when operators 
with purely absolutely continuous spectrum are 
similar to self-adjoint ones was discussed 
in \cite{KiselevFaddeev}.
Another approach is through the framework of extension theory 
for symmetric operators \cite{Kuzel05, Kiselev}.

In this section we follow a completely different approach
and succeed in reducing the question of quasi-self-adjointness
for the unbounded operator $-\Delta(A,B)$
to a simple check of the similarity of 
$\mathfrak{S}(k,A,B)$ to a unitary matrix.
The prize we pay is that the method applies to 
graphs with equal internal edge lengths only.
More specifically, throughout this section,
we always assume 
\begin{equation}\label{equal}
  a_i=a
  \qquad \mbox{for all} \qquad
  i\in\Ie.
\end{equation}

\subsection{From matrices to operators}
For any $(\abs{\Ie}\times\abs{\Ie})$-matrix 
$G(\Ie)=(G(\Ie)_{ji})$ defining an operator in $\Ke_{\Ie}^{-}$ 
we introduce the map 
\begin{eqnarray*}
\Phi_{G(\Ie)}\colon \He_{\Ie} \rightarrow \He_{\Ie}, & \displaystyle{\left(\Phi_{G(\Ie)} \psi\right)(x_j) = \sum_{i=1}^n G(\Ie)_{ji} \psi_i(x_j),} & j\in \Ie,
\end{eqnarray*}
where $n=\abs{\Ie}$. 
Accordingly, for a $(\abs{\Ee}\times\abs{\Ee})$-matrix $G(\Ee)=(G(\Ee)_{ji})$ defining an operator in $\Ke_{\Ee}$ we introduce 
\begin{eqnarray*}
\Phi_{G(\Ee)}\colon \He_{\Ee} \rightarrow \He_{\Ee}, & \displaystyle{\left(\Phi_{G(\Ee)} \psi\right)(x_j) = \sum_{i=1}^m G(\Ee)_{ji} \psi_i(x_j),} & j\in \Ee,
\end{eqnarray*}
where $m=\abs{\Ee}$. 
These maps are well defined since the functions $\psi_i$ 
are defined on the $i$-th edge, which is identified with a half-line or an interval $[0,a]$, respectively, and hence they can be interpreted as functions on another half-line or interval $[0,a]$ as well.

For any $\psi\in \De$ let us also define
\begin{align*}
\underline{\psi}_{\Ee}
&= \{\psi_{e}(0)\}_{e\in\Ee},  
&\underline{\psi}_{\Ie,-}
&=  \{\psi_{i}(0)\}_{i\in\Ie}, 
&\underline{\psi}_{\Ie,+}
&=\{\psi_{i}(a_i)\}_{i\in\Ie},
\\
\underline{\psi}^{\prime}_{\Ee}
&= \{\psi_{e}^{\prime}(0)\}_{e\in\Ee}, 
&\underline{\psi}^{\prime}_{\Ie,-}
&=\{\psi_{i}^{\prime}(0)\}_{i\in\Ie}, 
&\underline{\psi}^{\prime}_{\Ie,+}
&=\{-\psi_{i}^{\prime}(a_i)\}_{i\in\Ie},
\end{align*}
and 
\begin{eqnarray*}
\underline{\psi}_{\Ie}= \underline{\psi}_{\Ie,-}\oplus\underline{\psi}_{\Ie,+}, & & \underline{\psi}^{\prime}_{\Ie}=\underline{\psi}^{\prime}_{\Ie,-}\oplus\underline{\psi}^{\prime}_{\Ie,+}.
\end{eqnarray*}
Finally, we set 
\begin{eqnarray*}
\De_{\Ee}:=\De \cap \He_{\Ee} &\mbox{and} & \De_{\Ie}:=\De \cap \He_{\Ie}.
\end{eqnarray*}

Here we collect basic properties of the maps 
$\Phi_{G(\Ie)}$ and $\Phi_{G(\Ee)}$.
\begin{proposition}\label{PhiG}
The maps $\Phi_{G(\Ee)}$ and $\Phi_{G(\Ie)}$ are linear. 
For matrices $G,H$ one has  $\Phi_G \circ \Phi_H= \Phi_{G\circ H}$. 
In particular, if $G(\Ie)$ or $G(\Ee)$ is invertible, then $\Phi_{G(\Ee)}$ respectively $\Phi_{G(\Ie)}$ is invertible with 
\begin{eqnarray*}
\left(\Phi_{G(\Ee)}\right)^{-1}=\Phi_{G(\Ee)^{-1}} & \mbox{and} &\left(\Phi_{G(\Ie)}\right)^{-1}=\Phi_{G(\Ie)^{-1}},
\end{eqnarray*}
respectively. Furthermore $\Phi_{G(\Ee)}$ maps $\De_{\Ee}$ to $\De_{\Ee}$ and $\Phi_{G(\Ie)}$ maps $\De_{\Ie}$ to $\De_{\Ie}$. For $\psi\in \De_{\Ee}$ one has \begin{eqnarray*}
\underline{\Phi_{G(\Ee)}\psi}_{\Ee} = G(\Ee)\, \underline{\psi}_{\Ee} &\mbox{and} & \underline{\left(\Phi_{G(\Ee)}\psi\right)}_{\Ee}^{\prime} = G(\Ee) \, \underline{\psi}^{\prime}_{\Ee}.
\end{eqnarray*}
For $\psi\in \De_{\Ie}$ one has 
\begin{eqnarray*}
\underline{\Phi_{G(\Ie)}\psi}_{\Ie} = 
\begin{bmatrix} 
G(\Ie)\underline{\psi}_{\Ie,-} 
\\ 
G(\Ie)\underline{\psi}_{\Ie,+} 
\end{bmatrix} 
&
\mbox{and} 
& 
\underline{\Phi_{G(\Ie)}\psi}_{\Ie}^{\prime} = 
\begin{bmatrix} 
G(\Ie)\underline{\psi}_{\Ie,-}^{\prime} 
\\ 
G(\Ie)\underline{\psi}_{\Ie,+}^{\prime} \end{bmatrix}.
\end{eqnarray*}
\end{proposition}

\subsection{The main result}
Taking advantage of the transformation of the boundary values one obtains
\begin{theorem}\label{thm1v2} 
Let $(\Ge,\au)$ be a finite metric graph with equal internal edge lengths, 
\ie~\eqref{equal} holds.
Let $A,B$ and $A^{\prime},B^{\prime}$ be linear maps in $\Ke$ such that 
\begin{eqnarray*}
A^{\prime}= G^{-1} A G &\mbox{and} & B^{\prime}= G^{-1} B G,
\end{eqnarray*}
for an invertible operator $G$ in $\Ke$ of the block diagonal form 
\begin{eqnarray}\label{G}
G=\begin{bmatrix} G(\Ee) & 0 & 0 \\ 0 & G(\Ie) & 0 \\ 0 & 0 & G(\Ie) \end{bmatrix}
\end{eqnarray}
with $G(\Ee)$ an invertible operator in $\Ke_{\Ee}$ and $G(\Ie)$ an invertible operator in $\Ke_{\Ie}^{-}$. Then the Laplacians $-\Delta(A,B)$ and $-\Delta(A^{\prime},B^{\prime})$ are similar to each other, \ie
$$\Phi_{G^{-1}} \Delta(A,B) \Phi_{G}= \Delta(A^{\prime},B^{\prime})$$
with similarity transform 
\begin{equation}\label{Phi}
\Phi_{G^{-1}}:=\Phi_{G(\Ee)^{-1}}\oplus \Phi_{G(\Ie)^{-1}}.
\end{equation}
\end{theorem}
\begin{proof}
Let $(\Ge,\au)$ be a metric graph with equal internal edge lengths and
\begin{eqnarray*}
A^{\prime}= G^{-1} A G &\mbox{and} & B^{\prime}= G^{-1} B G,
\end{eqnarray*}
where $G$ is of the block-diagonal form given in the theorem. In order to prove that $\Delta(A^{\prime},B^{\prime})=\Phi_{G^{-1}} \Delta(A,B) \Phi_{G}$ one has to show 
\begin{itemize}
\item[(a)]  $\Phi_{G^{-1}}$ maps $\Dom \Delta(A,B)$ to $\Dom \Delta (A^{\prime}, B^{\prime})$;
\item[(b)] $\Phi_{G^{-1}} \Delta(A,B) \Phi_{G}\psi= \Delta(A^{\prime},B^{\prime})\psi$, for $\psi\in \Dom (A^{\prime},B^{\prime})$.
\end{itemize}
Note that $\Phi_G$ and $\Phi_{G^{-1}}$ commute with $\Delta$, 
and therefore~(b) holds.  

It remains to show that $\Phi_{G^{-1}} (\Dom \Delta(A,B))= \Dom (\Delta(A^{\prime},B^{\prime}))$. Since $\Phi_{G^{-1}}$ commutes with $\Delta$, 
it follows also that 
$\Phi_{G^{-1}} \Ran \Delta(A,B)= \Ran \Delta(A^{\prime},B^{\prime})$.
Consequently, by Lemma~\ref{PhiG}, 
$\Phi_{G^{-1}}$ maps $\De$ to $\De$. 
If $\psi\in \Dom(\Delta(A,B))$, then $\psi\in \De$ 
and~\eqref{AB} holds.
Applying it to the function $\Phi_{G^{-1}} \psi$,
we get
\begin{equation*}
G^{-1}\left\{A G G^{-1}\underline{\psi} + B GG^{-1}\underline{\psi}^{\prime}\right\} = 0,
\end{equation*}
therefore 
$\Phi_{G^{-1}} (\Dom (\Delta(A,B))) 
\subset \Dom (\Delta(A^{\prime},B^{\prime}))$. 
The other way round, 
one proves analogously 
$\Phi_{G} (\Dom( \Delta(A^{\prime},B^{\prime}))) \subset \Dom (\Delta(A,B))$. 
Since $\Phi_{G^{-1}}$ is a bijection this proves the claim.
\end{proof}

The main result of this section is now the following
direct consequence of Theorem~\ref{thm1v2}. 
\begin{corollary}\label{thm1} 
Let $(\Ge,\au)$ be a finite metric graph with equal internal edge lengths, 
\ie~\eqref{equal} holds.
\begin{enumerate}
\item
Let $A,B$ be linear maps in $\Ke$ such that
\begin{eqnarray*}
\mathfrak{S}(k,A,B) = G^{-1} U G,
\end{eqnarray*}
for an invertible operator $G$ in $\Ke$ of the block diagonal form \eqref{G}
with $G(\Ee)$ an invertible operator in $\Ke_{\Ee}$ and $G(\Ie)$ an invertible operator in $\Ke_{\Ie}^{-}$. Then the Laplacians $-\Delta(A,B)$ and $-\Delta(A_U,B_U)$ with
\begin{eqnarray*}
A_{U}:= - \frac{1}{2} \left(U -\mathds{1}\right) & \mbox{and} & B_{U}:=  \frac{1}{2ik} \left(U +\mathds{1}\right).
\end{eqnarray*}
are similar to each other with the similarity transform given in \eqref{Phi}. In particular, if $U$ is unitary then $-\Delta(A,B)$ is similar to a self-adjoint Laplacian.
\item
Let $L,P$ and $L^{\prime},P^{\prime}$ define m-sectorial boundary conditions. Assume  furthermore that  there is an invertible operator $G$ in $\Ke$ of the block diagonal form \eqref{G} with $G(\Ee)$ an invertible operator in $\Ke_{\Ee}$ and $G(\Ie)$ an invertible operator in $\Ke_{\Ie}^{-}$ such that 
\begin{eqnarray*}
P= G^{-1}P^{\prime}G & \mbox{and} & L = G^{-1}L^{\prime}G.
\end{eqnarray*}
Then $-\Delta(P+L,P^{\perp})$ and $-\Delta(P^{\prime}+L^{\prime},(P^{\prime})^{\perp})$ are similar to each other with the similarity transform given in \eqref{Phi}. In particular, if $L^{\prime}$ is Hermitian then $-\Delta(P+L,P^{\perp})$ is similar to a self-adjoint Laplacian.
\end{enumerate}
\end{corollary}
\begin{proof}
Consider the boundary conditions
\begin{eqnarray*}
A_{\mathfrak{S}}:= - \frac{1}{2} \left(\mathfrak{S} -\mathds{1}\right) & \mbox{and} & B_{\mathfrak{S}}:=  \frac{1}{2ik} \left(\mathfrak{S} +\mathds{1}\right),
\end{eqnarray*}
and $k>0$ such that $A_{\mathfrak{S}}+ikB_{\mathfrak{S}}$ is invertible, where $\mathfrak{S}:=\mathfrak{S}(k,A,B)$. These are equivalent to the boundary conditions defined by $A,B$. By assumption there is an invertible operator $G$ in $\Ke$ such that 
\begin{eqnarray*}
A_{\mathfrak{S}}= G^{-1} A_U G &\mbox{and} & B_{\mathfrak{S}}= G^{-1} B_U G.
\end{eqnarray*}
Applying Theorem \ref{thm1v2} proves the claim. 
For m-sectorial boundary conditions the proof is analogous.
\end{proof}

\begin{remark}
Corollary~\ref{thm1} can be alternatively proven by using the resolvent formula given in Proposition~\ref{prop:res} by proving the similarity of the resolvents where the similarity transforms are given by means of $\Phi_G$.  
\end{remark}

\subsection{Application to star graphs}
Theorem~\ref{thm1v2} simplifies in the case of star graphs.
Here a \emph{non-compact star graph} is a graph with $\Ie=\emptyset$, 
and a \emph{compact star graph with equal edge lengths} 
is a graph with $\Ee=\emptyset$ and $a_i=a$ for all $i\in \Ie$ 
such that $\partial_- (i)=\partial_- (i^{\prime})$ 
for any $i,i^{\prime}\in \Ie$ and $\partial_+ (i)\neq\partial_+ (i^{\prime})$ whenever $i \neq i^{\prime}$. 

For a non-compact star graph consider the operator $-\Delta(A,B)$ where $A,B$ are linear maps in $\Ke$. Two operators $-\Delta(A,B)$ and $-\Delta(A^{\prime},B^{\prime})$ are similar whenever there exists an invertible operator $G$ in $\Ke_{\Ee}$ such that 
\begin{eqnarray*}
A^{\prime}= G^{-1} A G &\mbox{and} & B^{\prime}= G^{-1} B G.
\end{eqnarray*}
For the case of regular boundary conditions one has to check only if the matrices $\mathfrak{S}(k,A,B)$ and $\mathfrak{S}(k,A^{\prime},B^{\prime})$ are similar to each other.

In order to have an equally simple criterion for a compact star graph, 
one can consider $-\Delta(A,B)$ with boundary conditions of the form 
\begin{eqnarray}\label{bc:cg}
A=\begin{bmatrix} A^- & 0 \\ 0 & A^+ \end{bmatrix} & \mbox{and} & B=\begin{bmatrix} B^- & 0 \\ 0 & B^+ \end{bmatrix},
\end{eqnarray}
where $A^-,B^-$ are arbitrary linear maps in $\Ke_{\Ie}^-$, and $A^+= a^+ \mathds{1}_{\Ke_{\Ie}^+}$ and $B^+= b^+ \mathds{1}_{\Ke_{\Ie}^+}$ with $a^+,b^+\in \C$. 

Let $A^-,B^-$ and $A^{\prime-},B^{\prime-}$ be linear maps in $\Ke^-$ such that 
\begin{eqnarray*}
A^{\prime-}= G^{-1} A^- G &\mbox{and} & B^{\prime-}= G^{-1} B^- G,
\end{eqnarray*}
for an invertible linear operator $G$ in $\Ke^-$. Consider boundary  conditions $A^{\prime},B^{\prime}$ of the form \eqref{bc:cg} defined by $A^{\prime-},B^{\prime-}$ and $a^+,b^+\in\C$, and $A,B$ also of the form \eqref{bc:cg} defined by $A^{-},B^{-}$ and the same numbers $a^+,b^+\in\C$. Then $-\Delta(A,B)$ and $-\Delta(A^{\prime},B^{\prime})$ are similar to each other with similarity transform $\Phi_{G^{-1}}$. Again, for the case of regular boundary conditions one has to check only if the matrices $\mathfrak{S}(k,A,B)$ and $\mathfrak{S}(k,A^{\prime},B^{\prime})$ are similar to each other. For taking into account only the boundary conditions at the central vertex it is crucial to impose identical boundary conditions at all endpoints.

\begin{Example}[Special case of Example~\ref{gsgnsgn}]
Consider Example \ref{gsgnsgn} 
for $\abs{\Ee}=3$, with $\abs{\Ee_-}=1$, $\abs{\Ee_+}=2$. Note that
\begin{eqnarray*}
\mathfrak{S}(k, A,B)= G 
\begin{bmatrix}
1 & 0 & 0\\ 0 & -1 & 0 \\ 0 & 0 & -1
\end{bmatrix} G^{-1} & \mbox{with } G=\begin{bmatrix}
1 & -1 & 1\\ 1 & 1 & 0 \\ 1 & 0 & 1
\end{bmatrix}.
\end{eqnarray*}
Hence, by Corollary \ref{thm1}, the operator $-\Delta(A,B)$ is unitarily equivalent to a self-adjoint Laplacian, 
namely to the direct sum of two Neumann Laplacians 
and one Dirichlet Laplacian on the half-line.
\end{Example}
\begin{Example}[Star graph with both essential and discrete spectra]
Consider a star graph with only two external edges and the m-sectorial boundary conditions defined by 
\begin{eqnarray*}
P=0 & \mbox{and} & L=\begin{bmatrix} 0 & 2 \\ 1/2 & 0 \end{bmatrix},
\end{eqnarray*}
that is $2\psi_2(0)+\psi_1^{\prime}(0)=0$ and $1/2 \psi_1(0)+\psi_2^{\prime}(0)=0$. Since 
\begin{eqnarray*}
L=\begin{bmatrix} 1/2 & 0 \\ 0 & 1/4 \end{bmatrix} L^{\prime} \begin{bmatrix} 2 & 0 \\ 0 & 4 \end{bmatrix}, & \mbox{where } L^{\prime}=\begin{bmatrix} 0 & 1 \\ 1 & 0 \end{bmatrix},
\end{eqnarray*}
one has by Corollary~\ref{thm1} that $-\Delta(L,\mathds{1})$ is similar to the self-adjoint operator $-\Delta(L^{\prime},\mathds{1})$. Hence, the continuous spectrum of $-\Delta(L,\mathds{1})$ is $[0,\infty)$ and the point spectrum contains only the isolated simple eigenvalue $-1$.   
\end{Example}
%



\subsection{Application to Example~\ref{ex1}}\label{Sec.similar}
During our work we had in mind, as a guiding example, the class of point interactions defined at point zero on the intervals $(-L,L)$, $L \in (0,+\infty]$ by
\begin{eqnarray*}
\begin{bmatrix}
\psi(0+) \\ \psi^{\prime}(0+)
\end{bmatrix}
=\begin{bmatrix}
e^{i\tau}& 0 \\ 0 & e^{-i\tau}
\end{bmatrix}
\begin{bmatrix}
\psi(0-) \\ \psi^{\prime}(0-)
\end{bmatrix} & \mbox{for }\tau\in [0,\pi/2],
\end{eqnarray*}
which is also discussed in Example~\ref{ex1} above. Actually, this has been the starting point of our study, and now we are in the position to apply our results to reproduce some of the results known for it.

\subsubsection{Regular case}
Let $\tau \in [0,\pi/2)$.
For the Cayley transform
\begin{eqnarray*}
\mathfrak{S}(A_{\tau},B_{\tau},k)= - (A_{\tau}+ ik B_{\tau})^{-1}(A_{\tau}- ik B_{\tau}), 
\end{eqnarray*}
an explicit computation yields the diagonalisation
\begin{eqnarray*}
\frac{1}{\cos(\tau)} \begin{bmatrix} i \sin(\tau) & 1 \\ 1 & -i\sin(\tau) \end{bmatrix}= \frac{-1}{2\cos(\tau)}\begin{bmatrix} 1 & 1 \\ e^{-i\tau} & -e^{i\tau} \end{bmatrix} 
\begin{bmatrix} 1 & 0 \\ 0 & -1 \end{bmatrix} \begin{bmatrix} -e^{i\tau} & -1 \\ -e^{-i\tau} & 1 \end{bmatrix}.
\end{eqnarray*}
Hence, one has using $\diag \{1, -1\}=Q\mathfrak{S}(A_{0},B_{0},k)Q$ the similarity 
\begin{eqnarray*}
\mathfrak{S}(A_{\tau},B_{\tau},k)= G_{\tau}^{-1} Q \mathfrak{S}(A_{0},B_{0},k)Q G_{\tau},
\end{eqnarray*}
where
\begin{eqnarray*}
Q=\frac{1}{\sqrt{2}}\begin{bmatrix} 1 & 1 \\ 1 & -1\end{bmatrix} & \mbox{and} & G_{\tau}=\frac{i}{\sqrt{2\cos(\tau)}}\begin{bmatrix} -e^{i\tau} & -1 \\ -e^{-i\tau} & 1 \end{bmatrix}.
\end{eqnarray*}
From Corollary~\ref{thm1} it follows that the operator $-\Delta(A_{\tau},B_{\tau})$ is similar to the self-adjoint Laplacian $-\Delta(A_{0},B_{0})$, and the similarity transform is given by $\Phi_{QG_{\tau}}$:
\begin{eqnarray*}
\Delta(A_{0},B_{0})=\Phi_{(QG_{\tau})^{-1}}\Delta(A_{\tau},B_{\tau})\Phi_{QG_{\tau}}.
\end{eqnarray*}
In fact, $-\Delta(A_{0},B_{0})$ is the standard Laplacian on the real line. 

One can now compute a metric operator, 
\ie\ the operator $\Theta_{\tau}$ such that
\begin{eqnarray*}
\Delta(A_{\tau},B_{\tau})^{\ast}= \Theta_{\tau} \Delta(A_{\tau},B_{\tau}) \Theta_{\tau}^{-1}.
\end{eqnarray*}
Since~$Q$ is unitary, a metric is given by the formula
\begin{eqnarray*}
\Theta_{\tau}=\Phi_{(G_{\tau}^{\ast} G_{\tau})^{-1}}, 
& \mbox{where} \quad 
\displaystyle{(G_{\tau}^{\ast}G_{\tau})^{-1}}= \frac{1}{\cos(\tau)}\begin{bmatrix} 1 & i \sin(\tau) \\ -i\sin(\tau) & 1\end{bmatrix}. 
\end{eqnarray*}
We also have $\Theta_{\tau}^{-1}=\Phi_{G_{\tau}^{\ast} G_{\tau}}$, where
\begin{eqnarray*}
{G_{\tau}^{\ast} G_{\tau}}
= \frac{1}{\cos(\tau)}\begin{bmatrix} 1 & -i \sin(\tau) \\ i\sin(\tau) & 1\end{bmatrix}. 
\end{eqnarray*}
One can rewrite this as
\begin{eqnarray*}
\Theta_{\tau}=  \frac{1}{\cos(\tau)}\left[ \mathds{1} - i\sin(\tau ) M_{\sgn} \cP \right] & \mbox{and} & \Theta_{\tau}^{-1}=  \frac{1}{\cos(\tau)}\left[ \mathds{1} + i\sin(\tau ) M_{\sgn} \cP \right]. 
\end{eqnarray*}
Here the operator~$\cP$ interchanges the edges 
(therefore it corresponds in fact to the parity operator
in a quantum-mechanical interpretation of the model)  
and~$M_{\sgn}$ denotes the multiplication by $+1$ 
on the first edge and by $-1$ on the second edge 
(therefore, identifying the graph with the real line, 
$M_{\sgn}$ corresponds to the multiplication by $\sgn$). 
This is, up to a constant factor, 
the metric operator given in~\cite{Petr}. 

Considering the same boundary conditions at the central vertex on the compact star graph with two edges one obtains that the operators is similar to a self-adjoint Laplacian for any self-adjoint boundary condition imposed at both endpoints simultaneously, in particular for Dirichlet boundary conditions as considered in \cite[Sec.~4]{Petr}. In all cases a similarity transform is given by $\Phi_{QG_{\tau}}$ and a metric operator is given by $\Phi_{(G_{\tau}^{\ast}G_{\tau})^{-1}}$.

\subsubsection{Irregular case}\label{subsec:irr}
Let $\tau = \pi/2$. 
One has
\begin{eqnarray*}
\begin{bmatrix} 0 & 1 \\ 0 & 0   \end{bmatrix}=\frac{1}{2}A_{\pi/2}\begin{bmatrix} 1 & 1 \\ -i & i   \end{bmatrix} & \mbox{and} & \begin{bmatrix} 0 & 0 \\ 0 & 1    \end{bmatrix}=\frac{1}{2} B_{\pi/2}\begin{bmatrix} 1 & 1 \\ -i & i   \end{bmatrix}. 
\end{eqnarray*}
Hence, by Theorem~\ref{thm1v2}, 
the operator $-\Delta(A_{\pi/2},B_{\pi/2})$ on the star graph with only two external edges is unitarily equivalent to $-\Delta(A^{\prime},B^{\prime})$ with 
\begin{eqnarray*}
A^{\prime}=\begin{bmatrix} 0 & 1 \\ 0 & 0   \end{bmatrix} & \mbox{and} & B^{\prime}=\begin{bmatrix} 0 & 0 \\ 0 & 1   \end{bmatrix}.
\end{eqnarray*}
These boundary conditions are $\psi_2(0)=\psi^{\prime}_2(0)=0$, that is $-\Delta(A^{\prime},B^{\prime})$ is the direct sum of the minimal operator $-\Delta^0$ on one edge and the maximal operator $-\Delta$ on the other edge.   
Recall that $\sigma(-\Delta(A_{\pi/2}, B_{\pi/2}))=\C$.

\subsubsection{Irregular compact case}
Consider a compact star graph and let more generally $A^-,B^-$ define arbitrary irregular boundary conditions at the central vertex, and let $a^+,b^+$ with $\Rank (a^+,\, b^+)=1$ define boundary conditions at the end points, such that one obtains boundary conditions of the form~\eqref{bc:cg}. Hence there is a $\underline{\psi}\in\Ker A^- \cap \Ker B^-\neq \{0\}$ with $\norm{\underline{\psi}}=1$ and there is a unitary map in $\Ke$ mapping $\underline{\psi}$ to a unit vector $e_i$, $i\in\Ie$. The boundary conditions 
\begin{eqnarray*}
A^{\prime-}= A^-U & \mbox{and} & B^{\prime-}=B^-U
\end{eqnarray*}
define a unitarily equivalent operator, but the edge $i$ is decoupled from the rest of the graph and the operator on this edge has domain
\begin{eqnarray*}
\{\psi\in \De_j \mid a^+\psi(a) - b^+\psi^{\prime}(a)=0\}
\end{eqnarray*}
which defines by Proposition~\ref{prop:dimM} operator with entire $\C$ in the spectrum. This shows that also the operator defined on a compact star graph with only two internal edges of equal length where the boundary conditions at the central vertex are given by $A_{\tau},B_{\tau}$ and at the endpoint arbitrary regular boundary conditions are imposed has empty resolvent set. This reproduces some of the result form \cite[Prop.~6]{Petr}. 

\subsubsection{Relation to Example~\ref{sgnsgn}}
Consider the boundary conditions defined by
\begin{eqnarray*}
A=\begin{bmatrix} 1 & -1 \\ 0 & 0   \end{bmatrix} & \mbox{and} & B=\begin{bmatrix} 0 & 0 \\ 1 & -1   \end{bmatrix}.
\end{eqnarray*}
Then one obtains 
\begin{eqnarray*}
\frac{1}{\sqrt{2}}AU=\begin{bmatrix} 1 & 0 \\ 0 & 0   \end{bmatrix}, &  \frac{1}{\sqrt{2}}BU=\begin{bmatrix} 0 & 0 \\ 1 & 0   \end{bmatrix}, & \mbox{where } U=\frac{1}{\sqrt{2}}\begin{bmatrix} 1 & 1 \\ 1 & -1   \end{bmatrix},
\end{eqnarray*}
and $U$ maps $\Ker A \cap \Ker B $ to $\Span \{e_2\}$. These boundary conditions define on one edge the minimal operator $-\Delta^0$ and on the other edge the maximal operator $-\Delta$. For a compact star graph with these boundary conditions at the central vertex the same holds. 

Note that for $A_{\pi/2}, B_{\pi/2}$ one has
\begin{eqnarray*}
\begin{bmatrix} 1 & -1 \\ 0 & 0   \end{bmatrix}=A_{\pi/2}\begin{bmatrix} 1 & 0 \\ 0 & -i   \end{bmatrix} & \mbox{and} & \begin{bmatrix} 0 & 0 \\ 1 & -1    \end{bmatrix}=B_{\pi/2}\begin{bmatrix} 1 & 0 \\ 0 & -i   \end{bmatrix}, 
\end{eqnarray*}
hence the operator defined by $A_{\pi/2},B_{\pi/2}$ at the central vertex of a star graph with two edges of equal, possibly infinite, length is unitarily equivalent to the operator $-\sgn(x)\frac{d}{dx}\sgn(x)\frac{d}{dx}$, if in addition at the endpoint the same boundary conditions are imposed.

\subsection{Applications to self-adjoint Laplacians}\label{subsec:star}
Theorem~\ref{thm1v2} and its Corollary~\ref{thm1} 
can also be interestingly applied to self-adjoint Laplacians, 
in order to simplify the computation of the spectrum. 
Consider a compact star graph 
(see Figure~\ref{Fig}(a) for an example with three edges)
with standard (or Kirchhoff)
boundary condition at the central vertex $v$, where $\deg(v)=\nu$, \ie
\begin{align}\label{bc:st}
A_{\nu}= \left[
   \begin{array}{cccccc}
     1 & -1 & 0 &\cdots & 0 & 0 \\
     0 & 1 & -1 &\cdots & 0 & 0  \\
     0 & 0 & 1 &\cdots & 0 & 0  \\
     \vdots &\vdots  & \vdots & & \vdots  & \vdots \\
        0 & 0 & 0 &\cdots & 1 & -1  \\
   0 & 0 & 0 &\cdots & 0 & 0
   \end{array}
\right], && B_{\nu}= \left[
   \begin{array}{cccccc}
     0 & 0 & 0 &\cdots & 0 & 0 \\
     0 & 0 & 0 &\cdots & 0 & 0  \\
     0 & 0 & 0 &\cdots & 0 & 0  \\
     \vdots &\vdots  & \vdots & & \vdots  & \vdots \\
        0 & 0 & 0 &\cdots & 0 & 0  \\
   1 & 1 & 1 &\cdots & 1 & 1 
   \end{array}
\right].
\end{align}
It is known that $$\left[\mathfrak{S}(k,A_{\nu},B_{\nu})\right]_{ij}= \frac{2}{\deg(v)}- \delta_{ij},$$ see, \eg, \cite[Ex.~2.4]{VKRS2006}, 
and furthermore, it admits the representation
\begin{equation*}
\mathfrak{S}(k,A_{\nu},B_{\nu}) = P_{\nu}^{\perp} -P_{\nu}, 
\end{equation*}
where $P_{\nu}$ is the orthogonal projector onto $\Ker B_{\nu}$ and its complementary projector $P_{\nu}^{\perp}=\mathds{1}-P_{\nu}$ is the orthogonal projector onto the space spanned by the vector $\{w_{\nu}\}_j=1$, $j=1,\ldots, \deg(v)$. Hence $\mathfrak{S}(k,A_{\nu},B_{\nu})$ has the eigenvalues $-1$ of multiplicity $\deg(v)-1$ and $+1$ of simple multiplicity. At the ends of the leads one imposes for example Dirichlet boundary conditions. 
Then by applying Theorem~\ref{thm1v2} one obtains that this operator is iso-spectral to a direct sum of operators on intervals. Namely, $\deg(v)-1$ Dirichlet Laplacians on intervals $[0,a]$ and one Laplacian on $[0,a]$ with Dirichlet boundary condition at~$a$ and Neumann boundary condition at~$0$. 
This provides a complete picture of the spectrum. The spectrum is purely discrete and the solutions $k_n$, $n\in \N$, of $\sin(ka)=0$ yield eigenvalues $k_n^2$ of multiplicity $\deg(v)-1$ and the solutions $k_m$, $m\in \N$, of $\cos(ka)=0$ yield eigenvalues $k_m^2$ of multiplicity one.

\begin{figure}[h]
\begin{center}
\subfigure[A compact star graph]{
\begin{tikzpicture}[scale=0.8]
\fill[black] (0,0) circle (1ex);
\draw[->, black, very thick] (0,0) -- (0,1.5);
\draw[black, very thick] (0,1.5) -- (0,2.9);
\fill[black] (0,2.9) circle (1ex);
\draw[->, black, very thick] (0,0) -- (1,-1);
\draw[black, very thick] (1,-1) -- (2,-2);
\fill[black] (2,-2) circle (1ex);
\draw[->, black, very thick] (0,0) -- (-1,-1);
\draw[black, very thick] (-1,-1) -- (-2,-2);
\fill[black] (-2,-2) circle (1ex);
\end{tikzpicture}
} \ \ \ \ \  \
\subfigure[Cubic graph]{
\begin{tikzpicture}[scale=2.5]
    \coordinate (A1) at (0, 0);
    \coordinate (A2) at (0, 1);
    \coordinate (A3) at (1, 1);
    \coordinate (A4) at (1, 0);
    \coordinate (B1) at (0.3, 0.3);
    \coordinate (B2) at (0.3, 1.3);
    \coordinate (B3) at (1.3, 1.3);
    \coordinate (B4) at (1.3, 0.3);
    
    \coordinate (A1A2) at (0, 0.5);
    \coordinate (A1A4) at (0.5, 0);
    \coordinate (A3A2) at (0.5, 1);
    \coordinate (A3A4) at (1, 0.5);
    \coordinate (B2B3) at (0.8, 1.3);
    \coordinate (B3B4) at (1.3, 0.8);
    \coordinate (B4B1) at (0.8, 0.3);
    \coordinate (B1A1) at (0.15, 0.15);
    \coordinate (B2A2) at (0.15, 1.15);
    \coordinate (B3A3) at (1.15, 1.15);
    \coordinate (B4A4) at (1.15, 0.15);

    \draw[->, very thick] (A1) -- (0, 0.5);
    \draw[very thick] (0, 0.5) -- (A2);
    \draw[very thick] (A2) -- (A3);
    \draw[very thick] (A3) -- (A4);
    \draw[->, very thick] (A1) -- (0.5, 0);
    \draw[very thick] (0.5, 0) -- (A4);

    \draw[dashed] (0.15, 0.15) -- (B1);
    \draw[->, dashed] (A1) -- (0.15, 0.15);

    \draw[dashed] (B1) -- (0.3, 0.8);
    \draw[->, dashed] (B2) -- (0.3, 0.8);
    \draw[->, very thick] (B2) -- (B2A2);
    \draw[very thick] (B2A2) -- (A2);
    \draw[->, very thick] (B2) -- (B2B3);
    \draw[very thick] (B2B3) -- (B3);
    \draw[->, very thick] (A3) -- (B3A3);
    \draw[very thick] (B3A3) -- (B3);
    \draw[->, very thick] (A3) -- (A3A2);
    \draw[very thick] (A3A2) -- (A2);
    \draw[->, very thick] (A3) -- (A3A4);
    \draw[very thick] (A3A4) -- (A4);
    \draw[->, very thick] (B4) -- (B4A4);
    \draw[very thick] (B4A4) -- (A4);
    \draw[->, very thick] (B4) -- (B3B4);
    \draw[very thick] (B3B4) -- (B3);
    \draw[->, dashed] (B4) -- (B4B1);
    \draw[dashed] (B4B1) -- (B1);

\fill[black] (A1) circle (0.3ex);
\fill[black] (A2) circle (0.3ex);
\fill[black] (A3) circle (0.3ex);
\fill[black] (A4) circle (0.3ex);
\fill[gray] (B1) circle (0.3ex);
\fill[black] (B2) circle (0.3ex);
\fill[black] (B3) circle (0.3ex);
\fill[black] (B4) circle (0.3ex);

\end{tikzpicture}
}
\caption{Graphs considered in Subsection~\ref{subsec:star}}\label{Fig}
\end{center}
\end{figure}
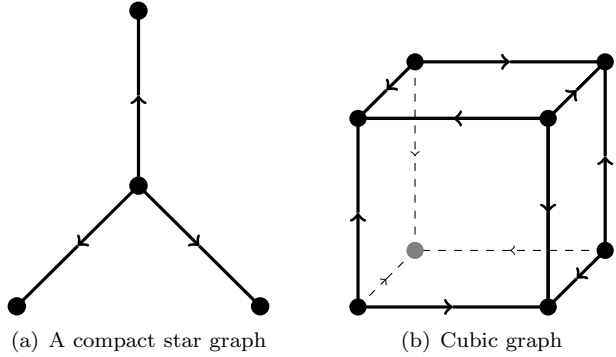

Consider as a further example a cube with equal edge lengths
(see Figure~\ref{Fig}(b)). 
Each vertex is a vertex of degree $\nu=3$ and one imposes at the vertices the standard boundary conditions \eqref{bc:st}. For this graph $\abs{\Ie}=12$ and the boundary conditions have the following block structure
\begin{eqnarray*}
A=\begin{bmatrix} A_{++} & 0 \\ 0 & A_{--}\end{bmatrix} &\mbox{and} & B=\begin{bmatrix} B_{++} & 0 \\ 0 & B_{--}\end{bmatrix}
\end{eqnarray*}
with respect to the decomposition $\Ke_{\Ie}=\Ke_{\Ie}^+ \oplus \Ke_{\Ie}^-$. Furthermore $A_{++}=A_{--}$ and $B_{++}=B_{--}$, where 
\begin{eqnarray*}
A_{++}=\begin{bmatrix} A_{3} & 0 & 0 & 0\\ 0 & A_{3} & 0 & 0 \\ 0 & 0& A_3 & 0 \\ 0 & 0 & 0 & A_3 \end{bmatrix} &\mbox{and} & B_{++}=\begin{bmatrix} B_{3} & 0 & 0 & 0\\ 0 & B_{3} & 0 & 0 \\ 0 & 0& B_3 & 0 \\ 0 & 0 & 0 & B_3 \end{bmatrix},
\end{eqnarray*}
where $A_3$ and $B_3$ are the matrices from \eqref{bc:st}. Applying as in the previous example the map $G_3$ that maps
\begin{eqnarray*}
G_3A_3 G_3^{-1}=\begin{bmatrix} 1 & 0 & 0 \\ 0 & 1 & 0 \\ 0 & 0 & 0 \end{bmatrix} &\mbox{and} & G_3B_3 G_3^{-1}=\begin{bmatrix} 0 & 0 & 0 \\ 0 & 0 & 0 \\ 0 & 0 & 1 \end{bmatrix} 
\end{eqnarray*}
one arrives at the conclusion that the Laplacian $-\Delta(A,B)$ with standard boundary conditions on the cube is unitarily equivalent to the direct sum of eight copies of the Dirichlet Laplacian on the interval and four copies of the Neumann Laplacian on the interval. This gives immediately the spectrum of the operator. In particular, the dimension of the  kernel of $-\Delta(A,B)$ is four, which follows already from the Euler characteristic of the graph which is $\abs{I}-\abs{V}=4$, see \cite[Thm.~20]{KFW}. 

More generally, one can consider any compact graph with equal edge lengths such that
\begin{itemize}
\item the degrees of all vertices agree and
\item one can find an orientation on the graph such that in each vertex there are either only incoming or only outgoing edges.
\end{itemize}
On such a graph one can impose of course also non-self-adjoint boundary conditions at the vertices, but one needs to impose at all vertices the same boundary conditions. If these boundary conditions are regular and the local scattering matrix is similar to a diagonal matrix then one can find an iso-spectral graph where all edges are decoupled and the calculation of the spectrum reduces to the case of intervals. 

For example, one can consider a compact graph with two edges $\Ie=\{i_1,i_2\}$ of equal length $a>0$ and two vertices $V=\{v_1,v_2\}$ with $\partial^{+}(i_1)=\partial^{+}(i_2)$ and $\partial^{-}(i_1)=\partial^{-}(i_2)$, that is a loop with two vertices. One imposes at each vertex boundary conditions $A_{\tau},B_{\tau}$ given in Example~\ref{ex1}, both with the same $\tau\in[0,\pi/2)$. Applying Corollary~\ref{thm1} delivers that the corresponding operator is similar to the Laplacian on the circle with arc length $2a$. Also, one can consider the cubic graph discussed above, where one imposes complex delta-couplings or the boundary conditions discussed in Example~\ref{gsgnsgn} at the vertices instead of the standard boundary conditions. 

\appendix
\section{Appendix}\label{Sec.App}
%
This appendix is devoted to the proof of Proposition~\ref{prop:res}.
\begin{definition}[{\cite[Def.~3.9]{KPS2008}}]\label{def:integralop}
The operator $\mathfrak{K}$ on the Hilbert space $\mathcal{H}$ is called \textit{integral operator} if for all $j,j^{\prime}\in \Ee \cup \Ie$ there are measurable functions $\mathfrak{K}_{j,j^{\prime}}(\cdot,\cdot)\colon I_j \times I_{j^{\prime}}\rightarrow \C$ with the following properties
\begin{enumerate}
\item $\mathfrak{K}_{j,j^{\prime}}(x_j,\cdot)\varphi_{j^{\prime}}\in L^1(I_{j^{\prime}})$ for almost all $x_j\in I_j$, 
\item $\psi = \mathfrak{K}\varphi$ with 
\begin{equation*}
\psi_j(x_j)= \sum_{j^{\prime}\in \Ee \cup \Ie} \int_{I_{j^{\prime}}} \mathfrak{K}_{j,j^{\prime}}(x_j,y_{j^{\prime}}) \varphi_{j^{\prime}}(y_{j^{\prime}}) d y_{j^{\prime}}.
\end{equation*}
\end{enumerate} 
The $(\Ie \cup \Ee) \times (\Ie \cup \Ee)$ matrix-valued function $(x,y) \mapsto \mathfrak{K}(x,y)$ with 
$$ [\mathfrak{K}(x,y)]_{j,j^{\prime}} = \mathfrak{K}_{j,j^{\prime}}(x_j,y_{j^{\prime}})$$
is called the \textit{integral kernel} of the operator $\mathfrak{K}$.
\end{definition}

In order to prove Proposition~\ref{prop:res}, 
we adapt the proof of \cite[Lem.~4.2]{VKRS2006}, 
where the resolvents of self--adjoint Laplace operators are considered, 
to the situation of more general regular boundary conditions.


By assumption the operator $\mathfrak{S}(k,A,B)$ is defined and $\mathds{1}-\mathfrak{S}(k,A,B)T(k;\au)$ is invertible for $k$ with $\Im k>0$. 
Hence the kernel $r_{\Me}(x,y;k)$ 
defined in Proposition~\ref{prop:res} is well-defined, 
and with $\Im k>0$ it defines a bounded operator $R_{\Me}(k)$ in $\He$ by 
\begin{eqnarray*}
R_{\Me}(k)\varphi = \int_{\Ge}r_{\Me}(\cdot,y;k)\varphi & \mbox{for } \varphi\in\He.
\end{eqnarray*}
In order to prove that $R_{\Me}(k)$ defines the resolvent operator, 
it suffices to check 
\begin{itemize}
\item[(i)] $R_{\Me}(k)\varphi\in \Dom(\Delta(A,B))$, for all $\varphi\in \He$,
\item[(ii)] $(-\Delta(A,B)-k^2)R_{\Me}(k)\varphi=\varphi$ for all $\varphi\in \He$ and
\item[(iii)] the symmetry relation $r_{\Me}(y,x;k)^{\ast}= r_{\Me^{\ast}}(x,y,-\overline{k})$.
\end{itemize}
The first two assertions prove that $(-\Delta(A,B)-k^2)R_{\Me}(k)=\mathds{1}_{\He}$ that is, $R_{\Me}(k)$ is the right inverse. By (iii) one proves that also $(-\Delta(A,B)^{\ast}-\overline{k}^2)R(k)_{\Me}^{\ast}=\mathds{1}_{\He}$, and taking the adjoints one obtains $R_{\Me}(k)(-\Delta(A,B)-k^2)\subset \mathds{1}_{\He}.$
This proves that $R_{\Me}(k)$ is also the left inverse.  

Using \cite[Lem.~4.2]{VKRS2006} and~\eqref{1ST},
one can also rewrite $r_{\Me}(x,y;k)$ as  
\begin{align*}
&r_{\Me}(x,y;k)= r^0(x,y;k)+r^1_{\Me}(x,y;k),\\ &r^1_{\Me}(x,y;k) =-\frac{i}{2k} \Phi(x,k) Z(k;A,B,\au)^{-1}(A-ikB) R_+(k;\au)^{-1}\Phi(y,k)^T.
\end{align*}
One can still prove (i) and (ii) whenever $Z(k;A,B,\au)$ is invertible proving that $R(k)$ defines the right inverse, but one cannot use the same proof for showing that the symmetry relation (iii) holds.

\begin{proof}[Proof of \emph{(i)}]
With $\psi= R_{\Me}(k)\varphi$, for $\varphi\in\He$ one has clearly $\psi\in\De$. Furthermore, set for brevity 
\begin{equation*}
G(k):= -Z(k; A,B,\au)^{-1}(A-ik B) R_+(k,\au)^{-1}.
\end{equation*}
Assume that $\varphi_j\in\He_j$ vanishes in a small neighbourhood of $x_j=0$ and, in addition, in a small
neighbourhood of $x_j=a_j$ if $j\in \Ie$. Then
\begin{equation*}
\int_{I_j} e^{ik \abs{x_j-y_j}}\varphi_j(y_j)dy_j= \int_{I_j} e^{-ik (x_j-y_j)}\varphi_j(y_j)dy_j
\end{equation*}
holds for sufficiently small $x_j\in I_j$, and for $x_j\in I_j$ sufficiently close to $a_j$ one has
\begin{equation*}
\int_{I_j} e^{ik \abs{x_j-y_j}}\varphi_j(y_j)dy_j= \int_{I_j} e^{ik (x_j-y_j)}\varphi_j(y_j)dy_j.
\end{equation*}
Therefore one obtains for the traces 
\begin{align*}
\underline{\psi}
&= \frac{i}{2k}R_+(k;\au)^{-1} \int_{\Ge} 
\Phi(y,k)^T \varphi(y) dy 
+ \frac{i}{2k} X(k;\au)G(k)\int_{\Ge} \Phi(y,k)^T \varphi(y) dy,
\\
\underline{\psi}^{\prime}
&=\frac{1}{2}R_+(k;\au)^{-1} \int_{\Ge} \Phi(y,k)^T \varphi(y) dy 
- \frac{1}{2} Y(k;\au)G(k)\int_{\Ge} \Phi(y,k)^T \varphi(y) dy. 
\end{align*}
Hence, 
\begin{align*}
A\underline{\psi} +B \underline{\psi}^{\prime}&= \frac{i}{2k}\left\{ (A-ikB)R_+(k;\au)^{-1}+ Z(k;A,B,\au)G(k)\right\}\int_{\Ge} \Phi(y,k)^T \varphi(y) dy\\ &=0.
\end{align*}
Thus $R_{\Me}(k)$ maps a dense subset of $\He$ to $\Dom(\Delta(A,B))$. By continuous continuation the claim follows for all $\varphi\in\He$ which proves (i). 
\end{proof}
\begin{proof}[Proof of \emph{(ii)}]
Assume that $\varphi_j\in C^{\infty}_0(I_j)$ for every $j\in \Ie\cup\Ee$. Since $\tfrac{i}{2k}e^{ik\abs{x-y}}$ defines the Green's function on the real line it follows that 
\begin{eqnarray*}
-\frac{i}{2k} \left(\frac{d^2}{dx_j^2}+k^2 \right)\int_{I_j}e^{ik\abs{x_j-y_j}}\varphi_j(y_j) dy_j = \varphi_j(x_j), & j\in \Ie \cup \Ee.
\end{eqnarray*}
Note that the remainder vanishes, and therefore one has proven 
the identity
$\left(-\Delta(A,B)-k^2\right)R_{\Me}(k)\varphi=\varphi$ 
for a dense subset of $\He$ 
and by continuous continuation the claim follows.
\end{proof}
\begin{proof}[Proof of \emph{(iii)}]
The relation $r^0(y,x;k)^{\ast}=r^0(x,y,-\overline{k})$ can be verified directly. For the remainder one obtains 
\begin{multline*}
r_{\Me}^1(y,x,k)^{\ast}
= \frac{i}{2(-\overline{k})} \Phi(x,-\overline{k}) 
R_+(-\overline{k};\au)^{-1}\mathfrak{S}(k,A,B)^{\ast}
\\
\times \left[\mathds{1}- 
T(-\overline{k};\au)\mathfrak{S}(k,A,B)^{\ast}\right]^{-1} 
R_+(-\overline{k};\au)^{-1}\Phi(y,-\overline{k})^T,
\end{multline*}
Note that
\begin{multline*}
\mathfrak{S}(k,A,B)^{\ast}\left[\mathds{1}-  T(-\overline{k};\au)\mathfrak{S}(k,A,B)^{\ast}\right]^{-1}
\\
=\left[\mathds{1}- \mathfrak{S}(k,A,B)^{\ast} T(-\overline{k};\au)\right]^{-1}\mathfrak{S}(k,A,B)^{\ast}
\end{multline*}
and 
$
\mathfrak{S}(k,A,B)^{\ast}= \mathfrak{S}(-\overline{k},A^{\prime},B^{\prime}),
$
where 
\begin{eqnarray*}
A^{\prime}:= - \frac{1}{2} \left(\mathfrak{S}(k,A,B)^{\ast} -\mathds{1}\right) & \mbox{and} & B^{\prime}:=  \frac{1}{-2i\overline{k}} \left(\mathfrak{S}(k,A,B)^{\ast} +\mathds{1}\right).
\end{eqnarray*}
From Proposition \ref{prop:adj} it follows that $r_{\Me}^1(y,x;k)^{\ast}=r^1_{\Me^{\ast}}(x,y;-\overline{k})$, and  therefore $R_{\Me}(k)^{\ast}=R_{\Me^{\ast}}(-\overline{k})$.
\end{proof}
%


\subsection*{Acknowledgment}
The first author would like to thank the Doppler Institute 
in Prague for the kind hospitality during his stay there
in spring 2013.
The work has been partially supported
by the project RVO61389005 and the GACR grant No.\ P203/11/0701. 
P.S. wishes to acknowledge the SCIEX Programme, the work has been conducted within the SCIEX-NMS Fellowship, project 11.263.


\end{document}